\documentclass[a4paper,fleqn,longmktitle]{cas-sc}

\usepackage[numbers]{natbib}
\usepackage{amsmath,amssymb,mathtools,amsthm}
\usepackage{tikz}
\usepackage{float}
\hypersetup{hypertexnames=false}
\usetikzlibrary{arrows.meta,positioning,calc}

\ExplSyntaxOn
\cs_if_exist:NF \vbox_unpack_clear:N
  { \cs_new_eq:NN \vbox_unpack_clear:N \vbox_unpack_drop:N }
\ExplSyntaxOff

\newtheorem{theorem}{Theorem}[section]
\newtheorem{proposition}[theorem]{Proposition}
\newtheorem{lemma}[theorem]{Lemma}
\newtheorem{corollary}[theorem]{Corollary}
\newtheorem{definition}[theorem]{Definition}
\newtheorem{remark}[theorem]{Remark}
\newtheorem{example}[theorem]{Example}

\DeclareMathOperator{\Inv}{Inv}
\DeclareMathOperator{\Adj}{Adj}

\DeclareMathOperator{\sgn}{sgn}

\newcommand{\calF}{\mathcal{F}}

\newcommand{\set}[1]{\left\{#1\right\}}

\newcommand{\symdiff}{\mathbin{\triangle}}

\begin{document}
\let\WriteBookmarks\relax
\def\floatpagepagefraction{1}
\def\textpagefraction{.001}

\shorttitle{Rational Weyl group elements of odd type D}
\shortauthors{Yutong Zhang and Yaoran Yang}

\title[mode=title]{Rational Weyl group elements of odd type D}

\author[1]{Yutong Zhang}
\cormark[1]
\ead{yutongzhang@stu.scu.edu.cn}
\credit{Conceptualization, Methodology, Formal analysis, Writing - original draft}

\author[1]{Yaoran Yang}
\ead{yangyaoran@stu.scu.edu.cn}
\credit{Formal analysis, Validation, Writing - review and editing}

\affiliation[1]{organization={School of Mathematics, Sichuan University},
            addressline={24 First Loop Road South Section I},
            city={Chengdu},
            postcode={610064},
            state={Sichuan},
            country={China}}

\cortext[1]{Corresponding author}

\begin{abstract}
Voloshyn introduced rational Weyl group elements in connection with rational normal forms on complex reductive groups and conjectured that their number in type $D_r$, for odd $r$, is $2^r-1$.  We prove a stronger structural statement.  For every odd integer $r\geq 5$, the rational elements of $W(D_r)$ are precisely the longest element $w_0$ and two explicitly described signed cyclic elements $c_I,d_I$ for each non-empty subset $I\subseteq\{1,\ldots,r-1\}$.  Consequently, the rationality graph $\Gamma(D_r)$ is obtained by gluing two explicitly labelled subset-toggle graphs at $w_0$; it has $2^r-1$ vertices, and its only vertices of valency one are $c_{\{1\}}$ and $d_{\{1\}}$.  The proof combines a two-level acyclic description of the root-poset graphs $\Gamma_{c_I}$ with a rigidity theorem for simple left multiplications of the signed cyclic family.  A self-contained simply-laced reflection-preservation lemma and a terminal-layer argument provide the descent step, while every forbidden one-step move from the family is excluded by an explicit loop or two-cycle.
\end{abstract}


\maketitle

\section{Introduction}
\label{sec:intro}

Let $G$ be a connected complex reductive algebraic group.  Fix opposite Borel subgroups $B_+,B_-$, let $N_+,N_-$ be their unipotent radicals, and put $H=B_+\cap B_-$.  Thus $W=N_G(H)/H$ is the Weyl group.  After choosing a positive system $\Pi_+$, Voloshyn studied rational decompositions
\begin{equation}
  g=N(g)B(g)\overline{u}N(g)^{-1},
  \qquad
  N:G\dashrightarrow N_-,\quad B:G\dashrightarrow B_+,
  \label{eq:intro-normal-form}
\end{equation}
where $\overline u\in N_G(H)$ represents $u\in W$ and the identity is required for generic $g\in G$.  He associated to each $u\in W$ a descending sequence
\begin{equation*}
  \nu_0(u)=u(\Pi_+)\cap \Pi_+,
  \qquad
  \nu_k(u)=u(\Adj \nu_{k-1}(u))\cap \Pi_+,
  \qquad k\geq 1,
\end{equation*}
and called $u$ \emph{rational} when the stable value $\nu(u)$ is empty \cite{Voloshyn2026}.  Equivalently, the oriented graph with vertex set $\nu_0(u)$ and arrows
\begin{equation}
  \alpha\longrightarrow \beta
  \quad\Longleftrightarrow\quad
  u^{-1}(\alpha)\leq \beta
  \label{eq:intro-arrow}
\end{equation}
has no directed cycle.  The rational Weyl group elements form the vertex set of another graph, denoted $\Gamma(W)$, in which two vertices are adjacent when they differ by left multiplication by a simple reflection.  For an irreducible root system, Voloshyn proved that $\Gamma(W)$ is connected and that it has more than one vertex precisely in types $A_r$ for $r\geq 2$, $D_r$ for odd $r$, and $E_6$; he also conjectured that in type $D_r$, $r$ odd, the number of rational elements is
\begin{equation}
  \#\{u\in W(D_r):u\text{ rational}\}=2^r-1.
  \label{eq:intro-count-conj}
\end{equation}
The computed values $31,127,511,2047,8191,32767$ for $r=5,7,9,11,13,15$ support \eqref{eq:intro-count-conj} \cite[Table~2]{Voloshyn2026}.  A second expectation in the same work is that the type-$D$ rationality graph has exactly two vertices of valency one; two such vertices had already been constructed explicitly in \cite[Proposition 4.15]{Voloshyn2026}.

The Coxeter-theoretic conventions used below are standard.  We use the usual
length function, simple reflections, Coxeter length criterion, and root-poset
terminology following \cite[Section~1.7]{Humphreys1990}, together with the
combinatorial Coxeter-group viewpoint of \cite[Chapter~4]{BjornerBrenti2005}.
Our root-system conventions agree with Bourbaki's presentation
\cite[Chapter~VI]{Bourbaki2002}, and the signed-permutation realization of
classical Weyl groups is the standard one described, for example, in Carter's
account of groups of Lie type \cite[Section~2.10]{Carter1972}.

The purpose of this paper is to prove a classification theorem that establishes \eqref{eq:intro-count-conj}, determines the graph $\Gamma(D_r)$, and verifies the valency-one expectation.  We use the standard signed-permutation model of $W(D_r)$.  Thus $W(D_r)$ is the subgroup of the signed symmetric group on an orthonormal basis $e_1,\ldots,e_r$ consisting of signed permutations with an even number of sign changes.  The positive and simple roots are
\begin{align}
  \Pi_+
  &=\set{e_i-e_j:1\leq i<j\leq r}
    \cup
    \set{e_i+e_j:1\leq i<j\leq r},
    \label{eq:positive-roots-intro} \\
  \Delta
  &=\set{\alpha_i=e_i-e_{i+1}:1\leq i\leq r-1}
    \cup
    \set{\alpha_r=e_{r-1}+e_r}.
    \label{eq:simple-roots-intro}
\end{align}
Write $\Pi_-=-\Pi_+$, and let $s_a=s_{\alpha_a}$ denote the reflection associated with $\alpha_a$.  For odd $r$, the longest element is $w_0(e_i)=-e_i$ for $1\leq i\leq r-1$, while $w_0(e_r)=e_r$.  Figure~\ref{fig:Dynkin-D} records this numbering.

For a subset $I=\{i_1<\cdots<i_k\}\subseteq\{1,\ldots,r-1\}$, $I\neq\varnothing$, define a cycle $p_I\in S_r$ by
\begin{equation}
  p_I(i_a)=i_{a+1}\;(1\leq a<k),
  \qquad
  p_I(i_k)=r,
  \qquad
  p_I(r)=i_1,
  \qquad
  p_I(j)=j\;(j\notin I\cup\{r\}).
  \label{eq:pI-intro}
\end{equation}
Equivalently, $p_I$ is the cycle $(i_1\ i_2\ \cdots\ i_k\ r)$ in the convention that $p_I(a)$ is the image of $a$.  We define two signed permutations $c_I,d_I\in W(D_r)$ by
\begin{align}
  c_I(e_j)&=-e_{p_I(j)}\quad(1\leq j\leq r-1),
  & c_I(e_r)&=e_{i_1},
  \label{eq:cI-intro} \\
  d_I(e_j)&=-e_{p_I(j)}\quad(j\in\{1,\ldots,r-1\}\setminus\{i_k\}),
  & d_I(e_{i_k})&=e_r,
  & d_I(e_r)&=-e_{i_1}.
  \label{eq:dI-intro}
\end{align}
The main theorem is the following.

\begin{theorem}
\label{thm:main-classification}
Let $r\geq 5$ be odd.  The rational Weyl group elements of type $D_r$ are precisely
\begin{equation*}
  \calF_r
  =\set{w_0}
   \cup
   \set{c_I,d_I:\varnothing\neq I\subseteq \{1,\ldots,r-1\}}.
\end{equation*}
Hence
\begin{equation}
  \#\calF_r=1+2(2^{r-1}-1)=2^r-1.
  \label{eq:main-count}
\end{equation}
Moreover $\Gamma(D_r)$ is determined by the following adjacency rules.  With the convention $c_\varnothing=d_\varnothing=w_0$, for $1\leq a\leq r-2$,
\begin{align}
  c_I\sim c_{I\symdiff\{a\}}
  &\quad\Longleftrightarrow\quad a+1\in I,
  \label{eq:main-c-horizontal} \\
  d_I\sim d_{I\symdiff\{a\}}
  &\quad\Longleftrightarrow\quad a+1\in I,
  \label{eq:main-d-horizontal}
\end{align}
and the two spin adjacencies are
\begin{equation}
  c_I\sim c_{I\symdiff\{r-1\}},
  \qquad
  d_I\sim d_{I\symdiff\{r-1\}}.
  \label{eq:main-spin-adj}
\end{equation}
There are no other edges.  In particular, $c_{\{1\}}$ and $d_{\{1\}}$ are the only vertices of valency one in $\Gamma(D_r)$.
\end{theorem}

The element $c_{\{1\}}$ is exactly the element denoted $C$ in \cite[Proposition 4.14]{Voloshyn2026}; its inverse is $d_{\{1\}}$.  Thus Theorem~\ref{thm:main-classification} simultaneously strengthens Voloshyn's type-$D$ counting conjecture and his valency-one expectation.  The argument is uniform in $r$ and rests on the following two steps.

First, for each $I$ the graph $\Gamma_{c_I}$ has a two-level form.  If $n_I(a)$ denotes the next element of $I\cup\{r\}$ after $a\in I$, then
\begin{equation*}
  \nu_0(c_I)=A_I\sqcup B_I,
  \qquad
  A_I=\set{e_b-e_{n_I(a)}\mid a\in I,\ a<b<n_I(a)},
  \qquad
  B_I=\set{e_{i_1}-e_q\mid i_1<q\leq r}.
\end{equation*}
Every arrow out of $A_I$ lands in $B_I$, and no arrow leaves $B_I$.  Hence $c_I$ is rational.  The elements $d_I$ follow from the diagram automorphism interchanging the two spin nodes.

Second, rationality is rigid under one-step descent.  Lemma~\ref{lem:voloshyn-descent}, based on the self-contained reflection-preservation Lemma~\ref{lem:reflection-preservation} and a terminal-layer argument that includes the case $m=0$, shows that every rational $u\neq w_0$ admits a simple reflection $s_a$ for which $s_a u$ is rational and $\ell(s_a u)=\ell(u)+1$.  We prove that whenever $v\in\calF_r$ and $s_a v$ is rational, then $s_a v\in\calF_r$.  Every forbidden one-step move from the signed cyclic family is excluded by an explicit loop or two-cycle in the root-poset graph \eqref{eq:intro-arrow}.  Induction on $\ell(w_0)-\ell(u)$ then gives the reverse inclusion.

\begin{figure}[pos=t]
\centering
\begin{tikzpicture}[scale=.82, every node/.style={font=\small}]
  \node (a1) at (0,0) {$1$};
  \node (a2) at (1.1,0) {$2$};
  \node (dots) at (2.2,0) {$\cdots$};
  \node (arm) at (3.5,0) {$r-2$};
  \node (ar1) at (4.7,.62) {$r-1$};
  \node (ar) at (4.7,-.62) {$r$};
  \draw (a1)--(a2)--(dots)--(arm);
  \draw (arm)--(ar1);
  \draw (arm)--(ar);
\end{tikzpicture}
\caption{The type-$D_r$ numbering used throughout.  The automorphism $\tau$ interchanges the two spin nodes $r-1$ and $r$.}
\label{fig:Dynkin-D}
\end{figure}

\section{Root-poset preliminaries and the signed cyclic family}
\label{sec:prelim}

Throughout, $r\geq 5$ is odd, and the type-$D_r$ root system is realized as in~\eqref{eq:positive-roots-intro}--\eqref{eq:simple-roots-intro}.  We use the usual root-poset order on $\Pi_+$: namely, $\rho\leq\eta$ if and only if $\eta-\rho\in\sum_{a=1}^r\mathbb Z_{\geq0}\alpha_a$.  If $x=(x_1,\ldots,x_r)\in\mathbb R^r$ is written as $x=\sum_{a=1}^r m_a(x)\alpha_a$, then
\begin{align*}
  m_a(x)&=x_1+\cdots+x_a,
  &&1\leq a\leq r-2, \\
  m_{r-1}(x)&=\frac{x_1+\cdots+x_{r-1}-x_r}{2},
  &
  m_r(x)&=\frac{x_1+\cdots+x_{r-1}+x_r}{2}.
\end{align*}
Consequently
\begin{equation}
  \rho\leq \eta
  \quad\Longleftrightarrow\quad
  m_a(\eta-\rho)\geq 0\quad(1\leq a\leq r).
  \label{eq:root-order-coeff}
\end{equation}

\begin{lemma}[Elementary order tests]
\label{lem:order-tests}
Let $1\leq a<b\leq r$ and $1\leq c<d\leq r$.  Then
\begin{equation}
  e_a-e_b\leq e_c-e_d
  \quad\Longleftrightarrow\quad
  c\leq a<b\leq d.
  \label{eq:minus-minus-order}
\end{equation}
For $1\leq a,b<r$,
\begin{equation}
  e_b+e_r\leq e_a+e_r
  \quad\Longleftrightarrow\quad
  a\leq b.
  \label{eq:plus-plus-order-1}
\end{equation}
For $1\leq a<r-1$ and $1\leq b<r$,
\begin{equation}
  e_b+e_r\leq e_a+e_{r-1}
  \quad\Longleftrightarrow\quad
  a\leq b.
  \label{eq:plus-plus-order-2}
\end{equation}
Finally, if $1\leq a<r$ and $1\leq c<d\leq r$, then
\begin{equation}
  e_a+e_r\nleq e_c-e_d.
  \label{eq:plus-not-minus}
\end{equation}
\end{lemma}

\begin{proof}
For a root $e_i-e_j$ with $i<j$, the simple-root coefficients are
\begin{equation*}
  m_t(e_i-e_j)=
  \begin{cases}
    1, & i\leq t<j,\quad 1\leq t\leq r-1,\\
    0, & \text{otherwise},
  \end{cases}
  \qquad
  m_r(e_i-e_j)=0.
\end{equation*}
Thus the coefficients of
$e_c-e_d-(e_a-e_b)$
are the differences of the characteristic functions of the intervals $[c,d)$ and $[a,b)$ on $\{1,\ldots,r-1\}$.  They are all non-negative exactly when
$[a,b)\subseteq [c,d),$
that is, exactly when $c\leq a<b\leq d$.  This proves \eqref{eq:minus-minus-order}; the coefficient $m_{r-1}$ is included in the interval test, while $m_r$ is zero on both roots.

For \eqref{eq:plus-plus-order-1},
$e_a+e_r-(e_b+e_r)=e_a-e_b$
is a non-negative sum of simple roots exactly when $a\leq b$.  Indeed, for $a\leq b$ it is $\alpha_a+\cdots+\alpha_{b-1}$, with the empty sum allowed when $a=b$, and the coefficient of $\alpha_b$ is negative when $a>b$.

For \eqref{eq:plus-plus-order-2},
$e_a+e_{r-1}-(e_b+e_r) =(e_a-e_b)+(e_{r-1}-e_r).$
If $a\leq b$, this is
$\sum_{t=a}^{b-1}\alpha_t+\alpha_{r-1},$
where the formula also covers $b=r-1$.  If $a>b$, the partial-sum coefficient at $t=b$ is negative, so the inequality fails.  Finally, $m_r(e_a+e_r)=1$, whereas $m_r(e_c-e_d)=0$ for all $c<d$; hence \eqref{eq:plus-not-minus} follows from \eqref{eq:root-order-coeff}.
\end{proof}

For $S\subseteq\Pi_+$, write $\Adj(S)=\set{\gamma\in\Pi_+\mid \gamma\leq\beta\text{ for some }\beta\in S}$.
The following definitions are those of Voloshyn \cite{Voloshyn2026}; they are included to fix the notation used in the proof.

\begin{definition}
\label{def:rational}
For $u\in W(D_r)$ define
\begin{align*}
  \nu_0(u)&=u(\Pi_+)\cap \Pi_+, \\
  \nu_k(u)&=u(\Adj(\nu_{k-1}(u)))\cap \Pi_+,
  \qquad k\geq 1.
\end{align*}
The oriented rationality graph of $u$, denoted here by $\Gamma_u$, has vertex set $\nu_0(u)$ and arrows
\begin{equation*}
  \alpha\to \beta
  \quad\Longleftrightarrow\quad
  u^{-1}(\alpha)\leq \beta.
\end{equation*}
We call $u$ \emph{rational} if $\nu_k(u)=\varnothing$ for all sufficiently large $k$.  The graph-theoretic characterization and the existence of a stable value are established in Lemma~\ref{lem:voloshyn-cycle}.
\end{definition}

The graph criterion and the reflection step are recalled in a form that is self-contained for the simply-laced root system considered here; compare \cite[Lemma~4.4]{Voloshyn2026}.

\begin{lemma}[Graph criterion]
\label{lem:voloshyn-cycle}
For $u\in W(D_r)$, the set $\nu_k(u)$ consists precisely of the vertices of $\Gamma_u$ from which a directed walk of length $k$ begins.  Consequently, the sequence $(\nu_k(u))_{k\geq0}$ is descending and stabilizes; denote its stable value by $\nu(u)$.  The following conditions are equivalent: $u$ is rational; $\Gamma_u$ has no directed cycle; and $\nu(u)=\varnothing$.  In particular, if $\alpha\in\nu_0(u)$ and $u^{-1}(\alpha)\leq\alpha$, then $u$ is not rational.
\end{lemma}

\begin{proof}
The assertion about directed walks follows by induction on $k$.  It is tautological for $k=0$.  For $k\geq0$, a vertex $\alpha$ belongs to $\nu_{k+1}(u)$ exactly when $u^{-1}(\alpha)\leq\beta$ for some $\beta\in\nu_k(u)$; equivalently, there is an arrow $\alpha\to\beta$ followed by a directed walk of length $k$.  Thus $\nu_{k+1}(u)\subseteq\nu_k(u)$.

Since $\nu_0(u)$ is finite, the descending sequence stabilizes.  If $\Gamma_u$ is acyclic, every directed walk has length at most $\#\nu_0(u)-1$, so $\nu_k(u)=\varnothing$ for all sufficiently large $k$.  Conversely, a directed cycle can be traversed arbitrarily many times, and therefore $\nu_k(u)\neq\varnothing$ for every $k$.  This proves the equivalences.  The final assertion is the special case of a directed cycle of length one.
\end{proof}

For a root $\rho=\sum_{\delta\in\Delta}c_\delta(\rho)\delta$, write $c_\alpha(\rho)$ for the coefficient of a fixed simple root $\alpha$.

\begin{lemma}[Root subtraction in simply-laced type]
\label{lem:root-subtraction}
Let $\alpha\in\Delta$ and let $\rho,\sigma\in\Pi_+$ satisfy $\rho\leq\sigma$, $c_\alpha(\rho)=c_\alpha(\sigma)$, $\sigma-\alpha\in\Pi_+$, and $\rho\neq\alpha$.  Then $\rho-\alpha\in\Pi_+$.
\end{lemma}

\begin{proof}
Since $c_\alpha(\rho)=c_\alpha(\sigma)$, the difference $\sigma-\rho$ is a non-negative integral combination of simple roots distinct from $\alpha$.  In a simply-laced root system, $\langle\delta,\alpha^\vee\rangle\leq0$ for every simple root $\delta\neq\alpha$; hence $\langle\rho,\alpha^\vee\rangle\geq\langle\sigma,\alpha^\vee\rangle$.  All roots have the same length.  Hence $\lVert\sigma-\alpha\rVert^2=\lVert\sigma\rVert^2$ gives $2(\sigma,\alpha)=\lVert\alpha\rVert^2$, and therefore $\langle\sigma,\alpha^\vee\rangle=1$.  Since $\rho\neq\alpha$ and the root system is simply laced, one also has $\langle\rho,\alpha^\vee\rangle\leq1$.  Thus $\langle\rho,\alpha^\vee\rangle=1$, and $s_\alpha(\rho)=\rho-\alpha$ is a root.  Its simple-root coefficients are non-negative because $c_\alpha(\rho)=c_\alpha(\sigma)\geq1$.  Therefore $\rho-\alpha$ is positive.
\end{proof}

\begin{lemma}[Reflection preservation]
\label{lem:reflection-preservation}
Let $u\in W(D_r)$ be rational, and let $\alpha\in\Delta$ satisfy $u^{-1}(\alpha)>0$ and $u(\alpha)<0$.  Then $s_\alpha u$ is rational.
\end{lemma}

\begin{proof}
Set $w=s_\alpha u$ and suppose, for a contradiction, that $\Gamma_w$ contains a directed cycle.  Since $u^{-1}(\alpha)>0$, the root $\alpha$ belongs to $\nu_0(u)$, and left multiplication by $s_\alpha$ removes precisely this root from the positive image set.  More explicitly, $s_\alpha$ preserves $\Pi_+\setminus\{\alpha\}$ and $\Pi_-\setminus\{-\alpha\}$, while $u(\Pi_+)$ contains $\alpha$ but not $-\alpha$; hence $\nu_0(w)=s_\alpha(\nu_0(u)\setminus\{\alpha\})$.

Index the assumed cycle by $\mathbb Z/m\mathbb Z$ so that its vertices are $\theta_i=s_\alpha\gamma_i$, where $\gamma_i\in\nu_0(u)\setminus\{\alpha\}$, and its arrows are equivalent to
\begin{equation}
  u^{-1}(\gamma_{i-1})\leq s_\alpha\gamma_i
  \qquad(i\in\mathbb Z/m\mathbb Z).
  \label{eq:reflection-cycle-ineq}
\end{equation}
Put $I_-:=\{i:u^{-1}(\gamma_{i-1})\leq\gamma_i\}$ and $I_+:=(\mathbb Z/m\mathbb Z)\setminus I_-$.  For $i\in I_-$ there is an arrow $\gamma_{i-1}\to\gamma_i$ in $\Gamma_u$.

Fix $i\in I_+$.  Because $\gamma_i\neq\alpha$ and the root system is simply laced, $s_\alpha\gamma_i$ is one of $\gamma_i-\alpha$, $\gamma_i$, or $\gamma_i+\alpha$.  The first two alternatives, together with \eqref{eq:reflection-cycle-ineq}, would imply $u^{-1}(\gamma_{i-1})\leq\gamma_i$, contrary to $i\in I_+$.  Thus $s_\alpha\gamma_i=\gamma_i+\alpha$.  Write
\begin{equation}
  \gamma_i+\alpha-u^{-1}(\gamma_{i-1})=\lambda_i,
  \qquad \lambda_i\in\sum_{\delta\in\Delta}\mathbb Z_{\geq0}\delta.
  \label{eq:lambda-i}
\end{equation}
Since $u^{-1}(\gamma_{i-1})\nleq\gamma_i$, the coefficient of $\alpha$ in $\lambda_i$ is zero; otherwise $\lambda_i-\alpha$ would still be a non-negative combination of simple roots.  Hence $\lambda_i\in\sum_{\delta\neq\alpha}\mathbb Z_{\geq0}\delta$.

Set $\eta_i=\gamma_i-\lambda_i$.  Equation \eqref{eq:lambda-i} gives $u^{-1}(\gamma_{i-1})=\eta_i+\alpha$.  Apply Lemma~\ref{lem:root-subtraction} with $\rho=u^{-1}(\gamma_{i-1})$ and $\sigma=\gamma_i+\alpha$.  Its hypotheses hold: the two roots have the same $\alpha$-coefficient, $\rho\leq\sigma$ by \eqref{eq:reflection-cycle-ineq}, $\sigma-\alpha=\gamma_i$, and $\rho\neq\alpha$ because otherwise $\gamma_{i-1}=u(\alpha)<0$.  Therefore $\eta_i\in\Pi_+$.  Define $x_i=u(\eta_i)$.  Applying $u$ to $u^{-1}(\gamma_{i-1})=\eta_i+\alpha$ yields $x_i=\gamma_{i-1}-u(\alpha)>\gamma_{i-1}$ in the root order.  In particular, $x_i\in\Pi_+$ and hence $x_i\in\nu_0(u)$.  Moreover, $u^{-1}(x_i)=\eta_i\leq\gamma_i$, so $x_i\to\gamma_i$ is an arrow of $\Gamma_u$.

We now obtain a directed cycle in $\Gamma_u$.  If $I_+=\varnothing$, the arrows $\gamma_{i-1}\to\gamma_i$ already form one.  If $I_-=\varnothing$, then for every $i$ one has $u^{-1}(x_i)=\eta_i\leq\gamma_i<x_{i+1}$, so the arrows $x_i\to x_{i+1}$ form a closed directed walk and hence contain a directed cycle.  Finally, suppose that both sets are non-empty.  For each maximal cyclic block $a+1,\ldots,b$ contained in $I_+$ and preceded by $a\in I_-$, the inequalities above give the directed walk
\[
  \gamma_{a-1}\longrightarrow x_{a+1}\longrightarrow x_{a+2}
  \longrightarrow\cdots\longrightarrow x_b\longrightarrow\gamma_b.
\]
Indeed, the first arrow follows from $u^{-1}(\gamma_{a-1})\leq\gamma_a<x_{a+1}$, the intermediate arrows from $u^{-1}(x_j)=\eta_j\leq\gamma_j<x_{j+1}$, and the final arrow from $u^{-1}(x_b)=\eta_b\leq\gamma_b$.  For every $i\in I_-$ with $i+1\in I_-$, retain the direct arrow $\gamma_{i-1}\to\gamma_i$.  For an index $a\in I_-$ immediately preceding a block of $I_+$, replace that direct arrow and the missing transitions across the block by the displayed walk.  These pieces have matching endpoints and concatenate around the cyclic index set to a closed directed walk.  Every closed directed walk contains a directed cycle, contradicting the rationality of $u$.
\end{proof}

\begin{lemma}[Descent]
\label{lem:voloshyn-descent}
Let $u\in W(D_r)$ be rational and $u\neq w_0$.  Then there is a simple root $\alpha_a$ such that $u^{-1}(\alpha_a)>0$ and $u(\alpha_a)<0$, the element $s_a u$ is rational, and $\ell(s_a u)=\ell(u)+1$.
\end{lemma}

\begin{proof}
If $\nu_0(u)=\varnothing$, then $u$ sends every positive root to a negative root.  Thus $\ell(u)=\#\Pi_+=\ell(w_0)$, and the uniqueness of the longest element gives $u=w_0$, contrary to the hypothesis.  Hence $\nu_0(u)\neq\varnothing$.  Since $u$ is rational and the sequence $\nu_k(u)$ is descending, there is a largest integer $m\geq0$ for which $\nu_m(u)\neq\varnothing$; then $\nu_{m+1}(u)=\varnothing$.  The value $m=0$ occurs when $\Gamma_u$ has vertices but no arrows.

Choose $\beta\in\nu_m(u)$.  Because $\nu_m(u)\subseteq\nu_0(u)$, the root $u^{-1}(\beta)$ is positive.  Write $\beta=\sum_b c_b\alpha_b$ with $c_b\in\mathbb Z_{\geq0}$.  If $u^{-1}(\alpha_b)$ were negative for every $b$ with $c_b>0$, then $u^{-1}(\beta)=\sum_b c_bu^{-1}(\alpha_b)$ would be a nonzero element of the negative simple-root cone, contradicting its positivity.  Hence some $a$ satisfies $c_a>0$ and $u^{-1}(\alpha_a)>0$.

The inequality $c_a>0$ gives $\alpha_a\leq\beta$, so $\alpha_a\in\Adj(\nu_m(u))$.  If $u(\alpha_a)$ were positive, it would belong to $u(\Adj(\nu_m(u)))\cap\Pi_+=\nu_{m+1}(u)$, which is empty.  Thus $u(\alpha_a)<0$, and Lemma~\ref{lem:reflection-preservation} shows that $s_a u$ is rational.  Finally, the standard Coxeter length criterion gives $\ell(s_a u)=\ell(u)+1$ because $u^{-1}(\alpha_a)>0$; see \cite[Section~1.7]{Humphreys1990} or \cite[Chapter~4]{BjornerBrenti2005}.
\end{proof}

We shall use repeatedly the equivalent length criterion
$\ell(s_a u)=\ell(u)+1 \quad\Longleftrightarrow\quad u^{-1}(\alpha_a)>0.$

\subsection{The signed cyclic family}
\label{sec:family}

Set $[r-1]=\{1,2,\ldots,r-1\}$.  For every non-empty $I=\{i_1<\cdots<i_k\}\subseteq [r-1]$, let $p_I$ be the cycle in~\eqref{eq:pI-intro}.  We shall use the following notation throughout:
$n_I(i_a)=i_{a+1}\quad(1\leq a<k), \qquad n_I(i_k)=r.$
Thus $n_I(a)$ is the next marked point of $I\cup\{r\}$ after $a\in I$.  The elements $c_I,d_I$ are as in~\eqref{eq:cI-intro}--\eqref{eq:dI-intro}.  We adopt the convention $c_\varnothing=d_\varnothing=w_0$.  For $I\neq\varnothing$, this convention does not alter the definitions of the signs; it serves only to abbreviate the graph formulas.

\begin{lemma}
\label{lem:family-in-W}
For every non-empty $I\subseteq [r-1]$, both $c_I$ and $d_I$ belong to $W(D_r)$.
\end{lemma}

\begin{proof}
The underlying unsigned permutation is $p_I$ in both cases.  Since $r$ is odd, $r-1$ is even.  The element $c_I$ has negative signs at precisely the positions $1,\ldots,r-1$, and therefore has an even number of sign changes.  The element $d_I$ has positive signs at $i_k$ and negative signs at all positions in $\{1,\ldots,r\}\setminus\{i_k\}$; again the number of negative signs is $r-1$, which is even.  Hence both signed permutations lie in the even signed permutation group $W(D_r)$.
\end{proof}

\begin{lemma}[Faithfulness of the parametrization]
\label{lem:faithful-parametrization}
For non-empty subsets $I,J\subseteq[r-1]$, one has $p_I=p_J$ if and only if $I=J$.  Consequently, $c_I=c_J$ if and only if $I=J$, $d_I=d_J$ if and only if $I=J$, and $c_I\neq d_J$.  Moreover, none of these elements equals $w_0$.  Hence the family $\{w_0\}\cup\{c_I,d_I:\varnothing\neq I\subseteq[r-1]\}$ contains exactly $1+2(2^{r-1}-1)$ elements, and its $c$- and $d$-parts intersect only at the common convention $c_\varnothing=d_\varnothing=w_0$.
\end{lemma}

\begin{proof}
The unique non-trivial orbit of $p_I$ is $I\cup\{r\}$.  Thus $p_I=p_J$ forces equality of these orbits and hence $I=J$; the converse is immediate.  Equality of two $c$-elements or two $d$-elements therefore forces equality of their underlying permutations and hence of their index sets.  On the other hand, $c_I(e_r)=e_{\min I}$ whereas $d_J(e_r)=-e_{\min J}$, so $c_I\neq d_J$.  Finally, $p_I$ is non-trivial for every non-empty $I$, while the unsigned permutation underlying $w_0$ is the identity.  This proves all assertions.
\end{proof}

Let $\tau$ be the diagram automorphism of the based root system defined by
$\tau(e_j)=e_j\quad(1\leq j\leq r-1), \qquad \tau(e_r)=-e_r.$
Then
$\tau(\alpha_a)=\alpha_a\quad(1\leq a\leq r-2), \qquad \tau(\alpha_{r-1})=\alpha_r, \qquad \tau(\alpha_r)=\alpha_{r-1}.$
Thus $\tau$ preserves $\Pi_+$ and the root order.  It normalizes $W(D_r)$ inside the full automorphism group of the root system, and
\begin{equation}
  d_I=\tau c_I\tau^{-1},
  \qquad
  \tau s_{r-1}\tau^{-1}=s_r,
  \qquad
  \tau s_a\tau^{-1}=s_a\quad(1\leq a\leq r-2).
  \label{eq:tau-conjugation}
\end{equation}
Consequently $u$ is rational if and only if $\tau u\tau^{-1}$ is rational, because
\begin{equation}
  \nu_0(\tau u\tau^{-1})=\tau(\nu_0(u)),
  \qquad
  (\tau u\tau^{-1})^{-1}(\tau\alpha)\leq \tau\beta
  \Longleftrightarrow
  u^{-1}(\alpha)\leq \beta.
  \label{eq:tau-rationality}
\end{equation}
Thus it is enough to prove most assertions for $c_I$; the corresponding assertions for $d_I$ follow by applying $\tau$.

\begin{example}
\label{ex:D5-family}
For $r=5$ the complete $c$-half of the family is
{\small
\begin{equation*}
\begin{gathered}
  c_\varnothing=(-1,-2,-3,-4,5),\\
  c_{\{1\}}=(-5,-2,-3,-4,1),\quad
  c_{\{2\}}=(-1,-5,-3,-4,2),\quad
  c_{\{3\}}=(-1,-2,-5,-4,3),\quad
  c_{\{4\}}=(-1,-2,-3,-5,4),\\
  c_{\{1,2\}}=(-2,-5,-3,-4,1),\quad
  c_{\{1,3\}}=(-3,-2,-5,-4,1),\quad
  c_{\{1,4\}}=(-4,-2,-3,-5,1),\\
  c_{\{2,3\}}=(-1,-3,-5,-4,2),\quad
  c_{\{2,4\}}=(-1,-4,-3,-5,2),\quad
  c_{\{3,4\}}=(-1,-2,-4,-5,3),\\
  c_{\{1,2,3\}}=(-2,-3,-5,-4,1),\quad
  c_{\{1,2,4\}}=(-2,-4,-3,-5,1),\quad
  c_{\{1,3,4\}}=(-3,-2,-4,-5,1),\\
  c_{\{2,3,4\}}=(-1,-3,-4,-5,2),\quad
  c_{\{1,2,3,4\}}=(-2,-3,-4,-5,1).
\end{gathered}
\end{equation*}
}
Here $(a_1,\ldots,a_5)$ means $e_j\mapsto \sgn(a_j)e_{|a_j|}$.  The $d$-half is obtained from $d_I=\tau c_I\tau^{-1}$.  If $I\neq\varnothing$, the entry with absolute value $5$ and the fifth entry occur in distinct positions, and both signs are reversed.  If $I=\varnothing$, these two operations act on the same entry and cancel, so $d_\varnothing=c_\varnothing=w_0$.
\end{example}

\section{Rationality and the internal arrow calculus}
\label{sec:acyclic}

We now prove that every element of $\calF_r$ is rational.  Since $w_0(\Pi_+)\subseteq\Pi_-$, one has $\nu_0(w_0)=\varnothing$.  The main point is an explicit description of $\nu_0(c_I)$.

Fix $I=\{i_1<\cdots<i_k\}\neq\varnothing$.  For $a\in I$ put $n=n_I(a)$.  Since there is no element of $I$ strictly between $a$ and $n$, one has $p_I(a)=n$, while $p_I(b)=b$ for $a<b<n$.  Define two sets of positive roots
\begin{align*}
  A_I&=\set{e_b-e_{n_I(a)}:a\in I,\; a<b<n_I(a)},\\
  B_I&=\set{e_{i_1}-e_q:i_1<q\leq r}.
\end{align*}
The roots in $A_I$ are the roots coming from ordinary inversions of $p_I$ on the negative part $\{1,\ldots,r-1\}$; the roots in $B_I$ are the roots created by the unique positive sign at $e_r$.

\begin{proposition}
\label{prop:nu0-cI}
For $I\neq\varnothing$,
\begin{equation}
  \nu_0(c_I)=A_I\sqcup B_I.
  \label{eq:nu0-cI}
\end{equation}
Moreover, for $a\in I$, $a<b<n_I(a)$,
\begin{equation}
  c_I^{-1}(e_b-e_{n_I(a)})=e_a-e_b,
  \label{eq:preimage-A}
\end{equation}
and, for $q>i_1$, if $t\in\{1,\ldots,r-1\}$ is the unique index with $p_I(t)=q$, then
\begin{equation}
  c_I^{-1}(e_{i_1}-e_q)=e_t+e_r.
  \label{eq:preimage-B}
\end{equation}
\end{proposition}

\begin{proof}
Let $\beta\in\Pi_+$.  We first take $\beta=e_a-e_b$ with $1\leq a<b\leq r$.  If $b=r$, then
$c_I(e_a-e_r)=-e_{p_I(a)}-e_{i_1}\in \Pi_-$
for all $a<r$.  Hence no root of the form $e_a-e_r$ contributes to $\nu_0(c_I)$.  If $b<r$, then $c_I(e_a-e_b)=-e_{p_I(a)}+e_{p_I(b)}=e_{p_I(b)}-e_{p_I(a)}$, which is positive if and only if $p_I(b)<p_I(a)$.  We now determine these inversions explicitly.  If $a\notin I$, then $p_I(a)=a$, whereas for every $b>a$ with $b<r$ one has either $p_I(b)=b$ or $p_I(b)>b$; hence no inversion begins at $a$.  If $a\in I$ and $n=n_I(a)$, then $p_I(a)=n$.  For $a<b<n$, the index $b$ is not in $I$, so $p_I(b)=b<n=p_I(a)$ and $(a,b)$ is an inversion.  For $b=n<r$, one has $p_I(b)>n=p_I(a)$, and for $b>n$ one has $p_I(b)\geq b>n=p_I(a)$.  Thus the inversions are exactly the pairs with $a\in I$ and $a<b<n_I(a)$.  For such a pair the positive image is $e_b-e_{n_I(a)}$.  This gives $A_I$, and the equality $c_I(e_a-e_b)=e_b-e_{n_I(a)}$ also proves \eqref{eq:preimage-A}.

Now take $\beta=e_a+e_b$ with $1\leq a<b\leq r$.  If $b<r$, then
$c_I(e_a+e_b)=-e_{p_I(a)}-e_{p_I(b)}\in\Pi_-,$
so there is no contribution.  If $b=r$, then $c_I(e_a+e_r)=-e_{p_I(a)}+e_{i_1}=e_{i_1}-e_{p_I(a)}$.  This is positive exactly when $p_I(a)>i_1$.  The set $p_I(\{1,\ldots,r-1\})$ is $\{1,\ldots,r\}\setminus\{i_1\}$, so the positive images are exactly the roots $e_{i_1}-e_q$ with $q>i_1$.  This is $B_I$, and \eqref{eq:preimage-B} follows by taking the unique $t$ with $p_I(t)=q$.

The sets $A_I$ and $B_I$ are disjoint because no root in $A_I$ has first index $i_1$: if $a=i_1$, then $e_b-e_{n_I(i_1)}$ has $b>i_1$, while if $a>i_1$, then again the first index is larger than $i_1$.  Therefore \eqref{eq:nu0-cI} is a disjoint union.
\end{proof}

\begin{lemma}
\label{lem:no-B-out}
No arrow in $\Gamma_{c_I}$ starts at a vertex of $B_I$.
\end{lemma}

\begin{proof}
Let $\alpha=e_{i_1}-e_q\in B_I$.  By \eqref{eq:preimage-B},
$c_I^{-1}(\alpha)=e_t+e_r$
for some $t<r$.  Every element of $A_I\cup B_I$ is of the form $e_x-e_y$.  By \eqref{eq:plus-not-minus},
$e_t+e_r\nleq e_x-e_y$
for all $x<y$.  Hence there is no $\beta\in\nu_0(c_I)$ with $c_I^{-1}(\alpha)\leq \beta$, and therefore no outgoing arrow from $\alpha$.
\end{proof}

\begin{lemma}
\label{lem:no-A-to-A}
No arrow in $\Gamma_{c_I}$ starts at a vertex of $A_I$ and ends at a vertex of $A_I$.
\end{lemma}

\begin{proof}
Take two roots of $A_I$:
$\alpha=e_b-e_{n_I(a)}, \qquad \beta=e_d-e_{n_I(c)},$
where
$a,c\in I, \qquad a<b<n_I(a), \qquad c<d<n_I(c).$
By \eqref{eq:preimage-A}, an arrow $\alpha\to\beta$ would be equivalent to $e_a-e_b\leq e_d-e_{n_I(c)}$.  By \eqref{eq:minus-minus-order}, this is equivalent to $d\leq a<b\leq n_I(c)$.  But $c<d\leq a$ gives $c<a$, and $a<n_I(c)$ places the marked point $a\in I$ strictly between $c$ and the next element $n_I(c)$ of $I\cup\{r\}$.  This is impossible.
\end{proof}

\begin{proposition}
\label{prop:family-rational}
Every element of $\calF_r$ is rational.
\end{proposition}

\begin{proof}
The assertion for $w_0$ follows from $\nu_0(w_0)=\varnothing$.  Let $I\neq\varnothing$.  By Proposition~\ref{prop:nu0-cI}, the vertex set of $\Gamma_{c_I}$ is $A_I\sqcup B_I$.  By Lemma~\ref{lem:no-B-out}, vertices in $B_I$ have no outgoing arrows.  By Lemma~\ref{lem:no-A-to-A}, there is no arrow from $A_I$ to $A_I$.  Thus every directed walk in $\Gamma_{c_I}$ has length at most one: arrows may run from $A_I$ to $B_I$, and no arrow leaves $B_I$.  In particular $\Gamma_{c_I}$ is acyclic, so $c_I$ is rational by Lemma~\ref{lem:voloshyn-cycle}.  Finally $d_I=\tau c_I\tau^{-1}$, and rationality is invariant under $\tau$ by \eqref{eq:tau-rationality}.  Therefore $d_I$ is rational.
\end{proof}

The proof gives the uniform stabilization bound
$\nu_2(c_I)=\varnothing$ and $\nu_2(d_I)=\varnothing$ for every non-empty $I$.
For example, when $I=\{1\}$, one obtains
\begin{align*}
  \nu_0(c_{\{1\}})
  &=\set{e_1-e_q:2\leq q\leq r}
    \cup
    \set{e_b-e_r:2\leq b\leq r-1},\\
  \nu_1(c_{\{1\}})
  &=\set{e_b-e_r:2\leq b\leq r-1},
  \qquad
  \nu_2(c_{\{1\}})=\varnothing,
\end{align*}
which recovers the element constructed in \cite[Proposition 4.14]{Voloshyn2026}.

\subsection{Internal arrow calculus}
\label{sec:arrow-calculus}

The proof of rationality above used only the absence of arrows from $B_I$ and of arrows from $A_I$ to $A_I$.  For later reference, and to make the stabilization statement completely explicit, we record the exact arrows.  This also gives a useful independent check on the formulas for $\nu_1$ and on the length computations in Section~\ref{sec:lengths}.

Fix $I=\{i_1<\cdots<i_k\}\neq\varnothing$.  For $a\in I$ and $a<b<n_I(a)$, write $A(a,b)=e_b-e_{n_I(a)}\in A_I$ and $A(a,b)^-=c_I^{-1}A(a,b)=e_a-e_b$.  For $q$ with $i_1<q\leq r$, write
$B(q)=e_{i_1}-e_q\in B_I, \qquad B(q)^-=c_I^{-1}B(q)=e_{t(q)}+e_r,$
where $t(q)$ is the unique element of $[r-1]$ satisfying $p_I(t(q))=q$.  Thus the arrow relation is
$X\longrightarrow Y \quad\Longleftrightarrow\quad X^-\leq Y, \qquad X,Y\in A_I\sqcup B_I.$

\begin{proposition}
\label{prop:exact-arrows-cI}
For $I\neq\varnothing$, the only arrows in $\Gamma_{c_I}$ are
\begin{equation}
  A(a,b)\longrightarrow B(q)
  \quad\Longleftrightarrow\quad
  b\leq q,
  \label{eq:exact-arrow-A-to-B}
\end{equation}
where $a\in I$, $a<b<n_I(a)$, and $i_1<q\leq r$.  Consequently
\begin{equation*}
  \nu_1(c_I)=A_I,
  \qquad
  \nu_m(c_I)=\varnothing\quad(m\geq 2).
\end{equation*}
The same statement for $d_I$ is obtained by applying $\tau$.
\end{proposition}

\begin{proof}
Lemma~\ref{lem:no-B-out} excludes all arrows with source in $B_I$, and Lemma~\ref{lem:no-A-to-A} excludes all arrows from $A_I$ to $A_I$.  It remains to decide when $A(a,b)$ points to $B(q)$.  By definition, this condition is $e_a-e_b\leq e_{i_1}-e_q$.  Using \eqref{eq:minus-minus-order}, it is equivalent to $i_1\leq a<b\leq q$.  The first two inequalities already hold because $a\in I$, $i_1=\min I$, and $a<b$.  Hence the condition reduces to $b\leq q$, proving \eqref{eq:exact-arrow-A-to-B}.

By definition of the sequence $\nu_m$, the set $\nu_1(c_I)$ consists exactly of those vertices of $\nu_0(c_I)$ from which at least one arrow starts.  Formula \eqref{eq:exact-arrow-A-to-B} shows that every $A(a,b)$ has an outgoing arrow, for instance to $B(r)$, while no $B(q)$ has one.  Therefore $\nu_1(c_I)=A_I$.  Since the target of every arrow lies in $B_I$, and no arrow starts in $B_I$, there is no directed walk of length two; equivalently $\nu_m(c_I)=\varnothing$ for $m\geq 2$.  The $d_I$ statement follows from \eqref{eq:tau-rationality}.
\end{proof}

The exact arrow formula admits a matrix description.  Order the vertices of $A_I$ lexicographically by the pair $(a,b)$ and the vertices of $B_I$ increasingly by $q$, and use the convention that the $(X,Y)$ entry is $1$ precisely when $X\to Y$.  The adjacency matrix of $\Gamma_{c_I}$ then has block form
\begin{equation*}
  M_I=
  \begin{pmatrix}
    0 & R_I\\
    0 & 0
  \end{pmatrix},
  \qquad
  (R_I)_{(a,b),q}=
  \begin{cases}
    1,& b\leq q,\\
    0,& b>q,
  \end{cases}
\end{equation*}
where $(a,b)$ ranges over $a\in I$, $a<b<n_I(a)$ and $q$ ranges over $i_1<q\leq r$.  In particular
$M_I^2=0.$
Thus $M_I$ is square-zero.  Over any field $\Bbbk$, the ideal $(M_I)=\Bbbk M_I$ of $\Bbbk[M_I]$ is square-zero, and $\Bbbk[M_I]=\operatorname{span}_{\Bbbk}\{\operatorname{Id},M_I\}$, where $\operatorname{Id}$ is the identity matrix.  This algebra has dimension one when $M_I=0$ and dimension two otherwise.  The number of arrows is also explicit.  Summing \eqref{eq:exact-arrow-A-to-B} gives
\begin{align}
  \#\operatorname{Arr}(\Gamma_{c_I})
  &=\sum_{a\in I}\sum_{b=a+1}^{n_I(a)-1}\#\{q:i_1<q\leq r,b\leq q\}
  \notag\\
  &=\sum_{a\in I}\sum_{b=a+1}^{n_I(a)-1}(r-b+1).
  \label{eq:number-arrows-cI}
\end{align}
The same number is obtained for $\Gamma_{d_I}$.

\begin{corollary}
\label{cor:arrow-count-closed}
For $I=\{i_1<\cdots<i_k\}$ and $n_I(i_j)=i_{j+1}$ for $j<k$, $n_I(i_k)=r$, one has
\begin{equation*}
  \#\operatorname{Arr}(\Gamma_{c_I})
  =\frac12\sum_{j=1}^k
     \bigl(n_I(i_j)-i_j-1\bigr)
     \bigl(2r-i_j-n_I(i_j)+2\bigr).
\end{equation*}
\end{corollary}

\begin{proof}
For a fixed gap $i_j<b<n_I(i_j)$, put $g_j=n_I(i_j)-i_j-1$.  The contribution of that gap to \eqref{eq:number-arrows-cI} is
\begin{align*}
  \sum_{b=i_j+1}^{n_I(i_j)-1}(r-b+1)
  &=g_j(r+1)-\sum_{b=i_j+1}^{n_I(i_j)-1}b
  \\
  &=g_j(r+1)-\frac{g_j(i_j+1+n_I(i_j)-1)}{2}
  \\
  &=\frac{g_j(2r-i_j-n_I(i_j)+2)}{2}.
\end{align*}
Summing over the gaps proves the formula.
\end{proof}

\begin{example}
\label{ex:c1-arrow-count}
For $I=\{1\}$, the single gap has $n_I(1)=r$.  Hence
$\#\operatorname{Arr}(\Gamma_{c_{\{1\}}}) =\frac12(r-2)(r+1).$
Indeed, $e_b-e_r\longrightarrow e_1-e_q$ exactly when $2\leq b\leq r-1$ and $b\leq q\leq r$, so the arrows are indexed by the triangular set
$\{(b,q):2\leq b\leq r-1,\ b\leq q\leq r\}.$
For $r=5$ this gives nine arrows: four out of $e_2-e_5$, three out of $e_3-e_5$, and two out of $e_4-e_5$.
\end{example}

The exact internal arrow calculus also motivates the rigidity analysis under left multiplication.  If $\Inv(v)=\{\beta\in\Pi_+:v(\beta)\in\Pi_-\}$ and $|\gamma|$ denotes the positive root in $\{\gamma,-\gamma\}$, then $\Inv(s_av)\mathbin{\triangle}\Inv(v)=\{|v^{-1}(\alpha_a)|\}$.  Thus $v^{-1}(\alpha_a)$ is added when it is positive, whereas $-v^{-1}(\alpha_a)$ is removed when $v^{-1}(\alpha_a)$ is negative.  In signed one-line notation, $s_a$ for $1\leq a\leq r-1$ acts by the corresponding unsigned transposition of target labels, while $s_r$ acts by the signed transposition $e_{r-1}\mapsto-e_r$ and $e_r\mapsto-e_{r-1}$.  The coordinate analysis below determines exactly which of these operations preserve rationality.

\section{One-step rigidity and classification}
\label{sec:rigidity}

The reverse inclusion in Theorem~\ref{thm:main-classification} is proved by induction.  The induction step requires the following exact description of which simple left multiplications of a family element can remain rational.

\subsection{Simple reflections acting on the family}
\label{subsec:action}

Let $I\subseteq [r-1]$, and use the convention $c_\varnothing=d_\varnothing=w_0$.  For $1\leq a\leq r-2$, left multiplication by $s_a$ interchanges the target labels $a$ and $a+1$.  Thus, in terms of the cycle $p_I$, it swaps the labels $a,a+1$ in the image of $p_I$.  The following identities are immediate from \eqref{eq:pI-intro}--\eqref{eq:cI-intro}; they are recorded because they are the combinatorial skeleton of $\Gamma(D_r)$.

\begin{lemma}
\label{lem:admissible-action-c}
For $I\subseteq [r-1]$ and $1\leq a\leq r-2$,
\begin{equation}
  a+1\in I
  \quad\Longrightarrow\quad
  s_a c_I=c_{I\symdiff\{a\}}.
  \label{eq:sa-cI-admissible}
\end{equation}
Furthermore
\begin{equation}
  s_{r-1}c_I=c_{I\symdiff\{r-1\}}.
  \label{eq:srminusone-cI}
\end{equation}
By applying $\tau$ one obtains
\begin{equation}
  a+1\in I
  \quad\Longrightarrow\quad
  s_a d_I=d_{I\symdiff\{a\}}
  \quad(1\leq a\leq r-2),
  \qquad
  s_r d_I=d_{I\symdiff\{r-1\}}.
  \label{eq:dI-action}
\end{equation}
\end{lemma}

\begin{proof}
For $1\leq a\leq r-2$, let $\sigma_a=(a\ a+1)\in S_r$.  Since $s_a$ interchanges $e_a$ and $e_{a+1}$ and fixes the remaining basis vectors, the unsigned permutation underlying $s_ac_I$ is $\sigma_a p_I$, while the sign attached to each source is unchanged.

Assume that $a+1\in I$.  If $a\in I$, let $h$ be the predecessor of $a$ in the cycle $p_I$; thus $p_I(h)=a$ and $p_I(a)=a+1$.  The permutation $\sigma_ap_I$ sends $h$ to $a+1$ and fixes $a$, while it agrees with $p_I$ on every other source except for these two forced target changes.  Hence it removes $a$ from the unique non-trivial cycle and equals $p_{I\setminus\{a\}}=p_{I\symdiff\{a\}}$.  This argument also covers the case $a=\min I$, in which the predecessor $h$ is $r$.  If $a\notin I$, let $h$ be the predecessor of $a+1$ in $p_I$.  Then $p_I(h)=a+1$ and $p_I(a)=a$.  Consequently, $\sigma_ap_I$ sends $h$ to $a$ and $a$ to $a+1$, thereby inserting $a$ immediately before $a+1$ in the cycle.  Thus $\sigma_ap_I=p_{I\cup\{a\}}=p_{I\symdiff\{a\}}$.  Since the source signs of the $c$-family are unchanged, these permutation identities prove \eqref{eq:sa-cI-admissible}.

For the spin reflection $s_{r-1}$, let $\sigma=(r-1\ r)$.  This reflection interchanges $e_{r-1}$ and $e_r$ without changing signs, so the unsigned permutation underlying $s_{r-1}c_I$ is $\sigma p_I$.  If $r-1\notin I$ and $I\neq\varnothing$, the last part of the cycle $p_I$ is $i_k\mapsto r\mapsto i_1$, while $r-1$ is fixed.  After composing on the left with $\sigma$, this becomes $i_k\mapsto r-1\mapsto r\mapsto i_1$, so $\sigma p_I=p_{I\cup\{r-1\}}$.  If $r-1\in I$ and $I\neq\{r-1\}$, the last part is $i_{k-1}\mapsto r-1\mapsto r\mapsto i_1$; composition with $\sigma$ fixes $r-1$ and replaces this segment by $i_{k-1}\mapsto r\mapsto i_1$, so $\sigma p_I=p_{I\setminus\{r-1\}}$.  For $I=\{r-1\}$, the two transpositions cancel and $\sigma p_I=\mathrm{id}$; for $I=\varnothing$, one has $\sigma\mathrm{id}=p_{\{r-1\}}$.  The source signs again agree with those prescribed for the corresponding $c$-element, proving \eqref{eq:srminusone-cI} in every case.

Finally, conjugation by $\tau$ fixes $s_a$ for $a\leq r-2$, exchanges $s_{r-1}$ with $s_r$, and sends $c_I$ to $d_I$.  Applying $\tau$ to the two identities just proved gives \eqref{eq:dI-action}.
\end{proof}

The remaining simple reflections are not merely outside the family; they are not rational.  We first prove this for the $c$-half.

\subsection{Non-admissible ordinary reflections}
\label{subsec:ordinary-exclusion}

\begin{lemma}
\label{lem:ordinary-exclusion-c}
Let $I\subseteq [r-1]$ and let $1\leq a\leq r-2$.  If $a+1\notin I$, then $s_a c_I$ is not rational.
\end{lemma}

\begin{proof}
If $I=\varnothing$, then $c_I=w_0$.  Since $w_0(\alpha_a)=-\alpha_a$, we have
$s_a w_0(\alpha_a)=\alpha_a, \qquad (s_a w_0)^{-1}(\alpha_a)=\alpha_a,$
and $s_a w_0$ has a loop at $\alpha_a$.

Assume now $I\neq\varnothing$.  There are two cases.

If $a\notin I$, then both $a$ and $a+1$ are fixed by the unsigned cycle $p_I$.  Hence
$c_I(\alpha_a)=c_I(e_a-e_{a+1})=-e_a+e_{a+1}=-\alpha_a,$
and therefore
$s_a c_I(\alpha_a)=\alpha_a, \qquad (s_a c_I)^{-1}(\alpha_a)=\alpha_a.$
Thus $s_a c_I$ has a loop.

If $a\in I$ and $a+1\notin I$, put $m=n_I(a)$.  Then $m>a+1$, while the definition of $n_I(a)$ gives $p_I(a)=m$ and $p_I(a+1)=a+1$.  Hence
$c_I(e_a-e_{a+1})=-e_m+e_{a+1}=e_{a+1}-e_m,$
and left multiplication by $s_a$ gives
$s_a c_I(e_a-e_{a+1})=e_a-e_m.$
Thus, with $u=s_a c_I$ and $\beta=e_a-e_m$, one has
$u^{-1}(\beta)=e_a-e_{a+1}\leq e_a-e_m=\beta,$
where the inequality follows from \eqref{eq:minus-minus-order}.  Hence $u$ has a loop at $\beta$.

In all cases Lemma~\ref{lem:voloshyn-cycle} implies that $s_a c_I$ is not rational.
\end{proof}

Applying $\tau$ gives the identical statement for the $d$-half.

\begin{corollary}
\label{cor:ordinary-exclusion-d}
Let $I\subseteq [r-1]$ and $1\leq a\leq r-2$.  If $a+1\notin I$, then $s_a d_I$ is not rational.
\end{corollary}

\subsection{Non-admissible spin reflections}
\label{subsec:spin-exclusion}

For the $c$-half, the admissible spin reflection is $s_{r-1}$; the other spin reflection $s_r$ is excluded except at $w_0$, where it gives $d_{\{r-1\}}$.

\begin{lemma}
\label{lem:spin-exclusion-c}
If $I\subseteq [r-1]$ is non-empty, then $s_r c_I$ is not rational.
\end{lemma}

\begin{proof}
Write $I=\{i_1<\cdots<i_k\}$ and put $u=s_r c_I$.

First suppose $r-1\notin I$.  Let $b=i_k$.  Since $p_I(b)=r$ and $c_I(e_r)=e_{i_1}$,
$c_I(e_b+e_r)=-e_r+e_{i_1}=e_{i_1}-e_r.$
Applying $s_r$, which sends $e_r$ to $-e_{r-1}$, yields
$u(e_b+e_r)=e_{i_1}+e_{r-1}.$
By \eqref{eq:plus-plus-order-2}, because $i_1\leq b<r-1$,
$e_b+e_r\leq e_{i_1}+e_{r-1}.$
Therefore $u$ has a loop at $e_{i_1}+e_{r-1}$.

Now suppose $r-1\in I$ and $I\neq\{r-1\}$.  Let $b$ be the largest element of $I\setminus\{r-1\}$.  Then $p_I(b)=r-1$, and
$c_I(e_b+e_r)=-e_{r-1}+e_{i_1}.$
Since $s_r(e_{r-1})=-e_r$, one obtains
$u(e_b+e_r)=e_{i_1}+e_r.$
By \eqref{eq:plus-plus-order-1},
$e_b+e_r\leq e_{i_1}+e_r,$
and $u$ has a loop at $e_{i_1}+e_r$.

It remains to treat $I=\{r-1\}$.  In this case the two spin simple roots form a two-cycle.  Indeed,
\begin{align*}
  u(\alpha_{r-1})
  &=s_r c_I(e_{r-1}-e_r)
    =s_r(-e_r-e_{r-1})
    =e_{r-1}+e_r=\alpha_r,\\
  u(\alpha_r)
  &=s_r c_I(e_{r-1}+e_r)
    =s_r(e_{r-1}-e_r)
    =e_{r-1}-e_r=\alpha_{r-1}.
\end{align*}
Equivalently,
\begin{equation}
  u^{-1}(\alpha_{r-1})=\alpha_r\leq \alpha_r,
  \qquad
  u^{-1}(\alpha_r)=\alpha_{r-1}\leq \alpha_{r-1},
  \label{eq:spin-two-cycle-arrows}
\end{equation}
so $\Gamma_u$ contains the directed cycle
$\alpha_{r-1}\longrightarrow \alpha_r\longrightarrow \alpha_{r-1}.$
Thus $s_r c_I$ is not rational for every non-empty $I$.
\end{proof}

The diagram automorphism gives the dual spin exclusion.

\begin{corollary}
\label{cor:spin-exclusion-d}
If $I\subseteq [r-1]$ is non-empty, then $s_{r-1}d_I$ is not rational.
\end{corollary}

Combining the action and exclusion lemmas gives the promised one-step rigidity.

\begin{proposition}[One-step rigidity]
\label{prop:onestep-rigidity}
Let $I\subseteq[r-1]$, with $c_\varnothing=d_\varnothing=w_0$, and let $1\leq a\leq r$.  Then
\begin{align}
  s_ac_I\text{ is rational}
  &\Longleftrightarrow
  \bigl(a\leq r-2\text{ and }a+1\in I\bigr)
  \text{ or }a=r-1
  \text{ or }\bigl(a=r\text{ and }I=\varnothing\bigr),
  \label{eq:rational-neigh-c} \\
  s_ad_I\text{ is rational}
  &\Longleftrightarrow
  \bigl(a\leq r-2\text{ and }a+1\in I\bigr)
  \text{ or }a=r
  \text{ or }\bigl(a=r-1\text{ and }I=\varnothing\bigr).
  \label{eq:rational-neigh-d}
\end{align}
In every rational case the resulting element belongs to $\calF_r$.  More explicitly, the ordinary moves and the admissible spin moves are those in Lemma~\ref{lem:admissible-action-c} and its $d$-analogue, while
$s_{r-1}w_0=c_{\{r-1\}}, \qquad s_rw_0=d_{\{r-1\}}.$
Consequently, if $v\in\calF_r$ and $s_av$ is rational, then $s_av\in\calF_r$.
\end{proposition}

\begin{proof}
For $c_I$ and $a\leq r-2$, Lemma~\ref{lem:admissible-action-c} gives a rational family element when $a+1\in I$, whereas Lemma~\ref{lem:ordinary-exclusion-c} gives a loop when $a+1\notin I$.  The reflection $s_{r-1}$ is always admissible and satisfies $s_{r-1}c_I=c_{I\symdiff\{r-1\}}$.  If $I\neq\varnothing$, Lemma~\ref{lem:spin-exclusion-c} excludes $s_rc_I$; if $I=\varnothing$, direct calculation gives $s_rw_0=d_{\{r-1\}}$.  This proves \eqref{eq:rational-neigh-c}.  Conjugation by $\tau$ exchanges $c_I$ with $d_I$ and $s_{r-1}$ with $s_r$, proving \eqref{eq:rational-neigh-d}.  The displayed identities identify every rational neighbor as an element of $\calF_r$.
\end{proof}

\subsection{Local obstruction certificates}
\label{sec:certificates}

The induction proof in Section~\ref{sec:classification} needs only the one-step rigidity statement.  Nevertheless, it is useful to isolate the local obstruction certificates because they show that the exclusion is not a global enumeration phenomenon.  Each forbidden neighbouring element fails rationality for a visibly small reason: either a loop, or in one boundary case a two-cycle.

For an element $u\in W(D_r)$, call a positive root $\gamma$ a \emph{loop certificate} if
$u(\gamma)\in \Pi_+, \qquad u^{-1}(u(\gamma))=\gamma\leq u(\gamma).$
Equivalently, the vertex $u(\gamma)\in\nu_0(u)$ has a loop in $\Gamma_u$.  Similarly, call an ordered pair $(\gamma,\delta)$ of positive roots a \emph{two-cycle certificate} if
$u(\gamma),u(\delta)\in\Pi_+, \qquad \gamma\leq u(\delta), \qquad \delta\leq u(\gamma).$
Then $u(\gamma)\to u(\delta)\to u(\gamma)$ in $\Gamma_u$.

\begin{lemma}
\label{lem:certificate-nonrational}
If $u$ admits either a loop certificate or a two-cycle certificate, then $u$ is not rational.
\end{lemma}

\begin{proof}
A loop certificate is a directed cycle of length one in $\Gamma_u$.  A two-cycle certificate is a directed cycle of length two.  Lemma~\ref{lem:voloshyn-cycle} identifies rationality with acyclicity of $\Gamma_u$, so either certificate excludes rationality.
\end{proof}

For ordinary non-admissible reflections, the obstruction is always a loop and there are only two local patterns.  Let $u=s_a c_I$ with $1\leq a\leq r-2$ and $a+1\notin I$.  If $I=\varnothing$ or $a\notin I$, then the source root and loop vertex are both $\alpha_a$.  If $a\in I$, put $m=n_I(a)$; then $m>a+1$, the source root is $e_a-e_{a+1}$, and the loop vertex is $e_a-e_m$.  The computations are summarized by
\begin{equation*}
\begin{array}{c|c|c|c}
 \text{case}&\gamma&u(\gamma)&\gamma\leq u(\gamma)\\
 \hline
 I=\varnothing&\alpha_a&\alpha_a&\alpha_a\leq\alpha_a\\
 I\neq\varnothing,\,a\notin I&\alpha_a&\alpha_a&\alpha_a\leq\alpha_a\\
 I\neq\varnothing,\,a\in I,\,a+1\notin I&e_a-e_{a+1}&e_a-e_{n_I(a)}&e_a-e_{a+1}\leq e_a-e_{n_I(a)}.
\end{array}
\end{equation*}
The last inequality is the interval containment
$[a,a+1)\subseteq [a,n_I(a))$
inside the root poset for roots $e_x-e_y$.

For the wrong spin reflection $u=s_r c_I$, $I\neq\varnothing$, there are three patterns.  If $r-1\notin I$, the source root is $e_{i_k}+e_r$ and the positive image is $e_{i_1}+e_{r-1}$.  If $r-1\in I$ but $I\neq\{r-1\}$, with $b=\max(I\setminus\{r-1\})$, the source root is $e_b+e_r$ and the positive image is $e_{i_1}+e_r$.  If $I=\{r-1\}$, the spin roots form a two-cycle.  In formula form,
\begin{equation*}
\begin{array}{c|c|c|c}
 \text{case}&\gamma&u(\gamma)&\text{certificate}\\
 \hline
 r-1\notin I&e_{i_k}+e_r&e_{i_1}+e_{r-1}&e_{i_k}+e_r\leq e_{i_1}+e_{r-1}\\
 r-1\in I,\ I\neq\{r-1\}&e_b+e_r&e_{i_1}+e_r&e_b+e_r\leq e_{i_1}+e_r\\
 I=\{r-1\}&(\alpha_{r-1},\alpha_r)&(\alpha_r,\alpha_{r-1})&\alpha_{r-1}\leftrightarrow\alpha_r.
\end{array}
\end{equation*}
Here the first inequality is \eqref{eq:plus-plus-order-2}, the second is \eqref{eq:plus-plus-order-1}, and the last entry abbreviates the two arrows in~\eqref{eq:spin-two-cycle-arrows}.

Applying $\tau$ gives the complete obstruction catalogue for the $d$-half.  Thus every forbidden move from a vertex of $\calF_r$ is detected by one of the following positive roots:
\begin{equation}
\begin{array}{l|l}
 \text{forbidden element} & \text{loop vertex \(u(\gamma)\) or two-cycle pair}\\
 \hline
 \substack{s_a c_I,\;1\leq a\leq r-2\\ a+1\notin I}
 &\alpha_a\text{ or }e_a-e_{n_I(a)}\\
 \substack{s_r c_I\\ I\neq\varnothing}
 &e_{i_1}+e_{r-1},\ e_{i_1}+e_r,\text{ or }(\alpha_{r-1},\alpha_r)\\
 \substack{s_a d_I,\;1\leq a\leq r-2\\ a+1\notin I}
 &\tau(\alpha_a)\text{ or }\tau(e_a-e_{n_I(a)})\\
 \substack{s_{r-1}d_I\\ I\neq\varnothing}
 &\tau(e_{i_1}+e_{r-1}),\ \tau(e_{i_1}+e_r),\text{ or }(\alpha_r,\alpha_{r-1}).
\end{array}
\label{eq:complete-certificate-table}
\end{equation}
The notation in~\eqref{eq:complete-certificate-table} is deliberately redundant.  In the ordinary $d$-case, $\tau$ fixes $\alpha_a$ and fixes $e_a-e_{n_I(a)}$ unless $n_I(a)=r$, in which case $\tau(e_a-e_r)=e_a+e_r$; in the spin case it interchanges the two spin nodes.

\begin{remark}
\label{rem:certificate-strength}
The certificate catalogue is stronger than the statement that the excluded elements are outside the family.  For instance, a signed permutation could in principle leave the set $\calF_r$ but remain rational; Proposition~\ref{prop:onestep-rigidity} rules this out locally by displaying an actual directed cycle in the graph $\Gamma_u$.  This local obstruction is what makes the induction in Theorem~\ref{thm:main-classification} possible: Lemma~\ref{lem:voloshyn-descent} moves an arbitrary rational element upward by one unit of Coxeter length, and the catalogue prevents the descent path from leaving the signed cyclic family.
\end{remark}

\subsection{Classification and graph structure}
\label{sec:classification}

We can now prove the classification theorem.  Let
$N=\ell(w_0)=\#\Pi_+=r(r-1).$
Recall that $w_0$ is the unique element of length $N$.

\begin{proof}[Proof of Theorem~\ref{thm:main-classification}]
Proposition~\ref{prop:family-rational} gives
$\calF_r\subseteq \set{u\in W(D_r):u\text{ rational}}.$
We prove the reverse inclusion by induction on
$\delta(u)=N-\ell(u).$
If $\delta(u)=0$, then $u=w_0\in\calF_r$.  Suppose $\delta(u)>0$ and that every rational element $v$ with $\delta(v)<\delta(u)$ lies in $\calF_r$.  By Lemma~\ref{lem:voloshyn-descent}, there is a simple reflection $s_a$ such that
$v=s_a u$
is rational and
$\ell(v)=\ell(u)+1, \qquad \delta(v)=\delta(u)-1.$
By the induction hypothesis, $v\in\calF_r$.  Since $u=s_a v$ is rational, Proposition~\ref{prop:onestep-rigidity} implies $u\in\calF_r$.  Hence every rational element lies in $\calF_r$.

The count \eqref{eq:main-count} now follows from Lemma~\ref{lem:faithful-parametrization}: the $2^{r-1}-1$ non-empty subsets of $[r-1]$ give pairwise distinct elements $c_I$ and $d_I$, the two indexed families are disjoint, and none of their elements equals $w_0$.

It remains only to identify the edges.  By definition, two rational elements are adjacent in $\Gamma(D_r)$ if one is obtained from the other by left multiplication by a simple reflection.  Proposition~\ref{prop:onestep-rigidity} lists all such rational left multiplications, and Lemma~\ref{lem:admissible-action-c} identifies their images.  This gives exactly \eqref{eq:main-c-horizontal}--\eqref{eq:main-spin-adj}.  No additional edges exist.

For the valency statement, use \eqref{eq:main-c-horizontal}--\eqref{eq:main-spin-adj}.  If $I\neq\varnothing$, then
\begin{equation}
  \deg(c_I)=1+\#(I\cap\{2,3,\ldots,r-1\}),
  \qquad
  \deg(d_I)=1+\#(I\cap\{2,3,\ldots,r-1\}).
  \label{eq:degree-formula}
\end{equation}
The additional $1$ is the spin edge toggling $r-1$, and each element $a+1\in I$ with $1\leq a\leq r-2$ contributes the edge toggling $a$.  Thus $\deg(c_I)=1$ or $\deg(d_I)=1$ if and only if $I=\{1\}$.  Finally
$\deg(w_0)=2,$
because $w_0$ is adjacent only to $c_{\{r-1\}}$ and $d_{\{r-1\}}$.  Therefore the only valency-one vertices are $c_{\{1\}}$ and $d_{\{1\}}$.
\end{proof}

\begin{corollary}
\label{cor:voloshyn-count}
For $r\geq 5$ odd, Voloshyn's type-$D$ counting conjecture holds:
$\#\{u\in W(D_r):u\text{ rational}\}=2^r-1.$
Moreover, the two valency-one elements constructed in \cite[Proposition 4.15]{Voloshyn2026} are the only valency-one vertices of $\Gamma(D_r)$.
\end{corollary}

\begin{corollary}
\label{cor:two-halves}
The graph $\Gamma(D_r)$ is the union of the two isomorphic induced subgraphs $\Gamma_c=\Gamma(D_r)[\{w_0\}\cup\{c_I:I\neq\varnothing\}]$ and $\Gamma_d=\Gamma(D_r)[\{w_0\}\cup\{d_I:I\neq\varnothing\}]$, which intersect exactly in $w_0$.  The isomorphism $\Gamma_c\to\Gamma_d$ is induced by $\tau$.
\end{corollary}

\begin{proof}
Lemma~\ref{lem:faithful-parametrization} shows that the two vertex sets intersect only at $w_0$.  The adjacency rules show that every edge other than the two edges incident with $w_0$ remains within one indexed family, and \eqref{eq:tau-conjugation} gives the asserted isomorphism.
\end{proof}

\section{Combinatorial consequences}
\label{sec:recognition}

The classification can be reformulated as an intrinsic test on signed one-line notation.  This is useful because the definition of rationality refers to the whole root-poset graph, whereas the theorem reduces the decision problem to the cycle structure and signs of a signed permutation.

Let $u\in W(D_r)$ be written as
\begin{equation}
  u(e_j)=\varepsilon_j e_{\pi(j)},
  \qquad
  \varepsilon_j\in\{\pm1\},
  \qquad
  \pi\in S_r.
  \label{eq:signed-one-line}
\end{equation}
For a non-empty subset $I=\{i_1<\cdots<i_k\}\subseteq [r-1]$, the two cyclic families have the same unsigned permutation
$\pi=p_I=(i_1\ i_2\ \cdots\ i_k\ r)$
where the cycle is read in the source-to-target direction $j\mapsto p_I(j)$.  Thus the non-trivial orbit of $\pi$ is exactly $I\cup\{r\}$ and its entries appear in increasing order before returning from $r$ to $i_1$.

\begin{proposition}
\label{prop:recognition}
A signed permutation $u$ in~\eqref{eq:signed-one-line} is rational if and only if one of the following mutually exclusive alternatives holds:
\begin{enumerate}
\item $\pi=\mathrm{id}$ and
\begin{equation}
  (\varepsilon_1,\ldots,\varepsilon_r)=(-1,\ldots,-1,+1),
  \label{eq:recognition-w0}
\end{equation}
so $u=w_0$;
\item there is a non-empty $I=\{i_1<\cdots<i_k\}\subseteq [r-1]$ such that $\pi=p_I$ and
\begin{equation}
  \varepsilon_j=-1\quad(1\leq j\leq r-1),
  \qquad
  \varepsilon_r=+1,
  \label{eq:recognition-c}
\end{equation}
so $u=c_I$;
\item there is a non-empty $I=\{i_1<\cdots<i_k\}\subseteq [r-1]$ such that $\pi=p_I$ and
\begin{equation}
  \varepsilon_{i_k}=+1,
  \qquad
  \varepsilon_j=-1\quad(j\neq i_k),
  \label{eq:recognition-d}
\end{equation}
so $u=d_I$.
\end{enumerate}
\end{proposition}

\begin{proof}
The forward implication is Theorem~\ref{thm:main-classification}, written in the notation of \eqref{eq:signed-one-line}.  Conversely, each of the three alternatives gives exactly one of the elements listed in Theorem~\ref{thm:main-classification}; hence it is rational.  The alternatives are mutually exclusive because $w_0$ has trivial unsigned permutation, while $p_I$ is non-trivial for $I\neq\varnothing$, and the sign vector in~\eqref{eq:recognition-c} differs from the sign vector in~\eqref{eq:recognition-d}: in the former $\varepsilon_r=+1$, while in the latter $\varepsilon_r=-1$.
\end{proof}

The proposition gives a linear-time recognition procedure once the cycle of $\pi$ containing $r$ has been extracted.  Let $O_r(\pi)=\{r,\pi(r),\pi^2(r),\ldots\}$ be the $\pi$-orbit of $r$.  If $O_r(\pi)=\{r\}$, rationality is possible only for $w_0$.  Otherwise write $O_r(\pi)\setminus\{r\}=\{i_1<\cdots<i_k\}$.  Then $\pi=p_I$ if and only if $\pi(i_j)=i_{j+1}$ for $1\leq j<k$, $\pi(i_k)=r$, $\pi(r)=i_1$, and $\pi(j)=j$ for every $j\notin I\cup\{r\}$.  The sign tests \eqref{eq:recognition-c} and \eqref{eq:recognition-d} then distinguish the two families.  Equivalently, $u$ is rational precisely when it is $w_0$, or when $\pi=p_I$ for a nonempty $I$ and the sign vector is either $(-,\ldots,-,+)_r$ or has its unique positive entry at the source $\max I$.  Here $(-,\ldots,-,+)_r$ denotes the sign vector with its unique positive entry at the source $r$.

No sorting is required in an implementation: one may instead check $\pi(r)<\pi^2(r)<\cdots<\pi^k(r)<r$ and $\pi^{k+1}(r)=r$, together with the condition that every point outside this orbit is fixed.  Equivalently, when the unique non-trivial cycle is written from its smallest non-$r$ entry, it must have the form $(i_1\ i_2\ \cdots\ i_k\ r)$.

\begin{corollary}
\label{cor:no-extra-signs}
Let $u\in W(D_r)$ be rational and let $\pi$ be its unsigned permutation.  If $\pi\neq\mathrm{id}$, then $\pi$ has exactly one non-trivial cycle, that cycle contains $r$, and all entries outside the cycle are fixed with negative sign.  Moreover $u$ has exactly one positive sign.
\end{corollary}

\begin{proof}
This is immediate from Proposition~\ref{prop:recognition}.  For $c_I$, the unique positive sign is at the source $r$.  For $d_I$, the unique positive sign is at the source $\max I$.  Outside $I\cup\{r\}$, the underlying permutation is fixed and the sign is negative in both cases.
\end{proof}

\begin{corollary}
\label{cor:unsigned-count}
The unsigned permutations that occur among rational elements are precisely
$\mathrm{id} \quad\text{and}\quad p_I\quad(\varnothing\neq I\subseteq [r-1]).$
There are $2^{r-1}$ such unsigned permutations.  The identity has one rational lift, and every non-identity $p_I$ has exactly two rational signed lifts.
\end{corollary}

\begin{proof}
The first statement is the unsigned part of Proposition~\ref{prop:recognition}.  Since there are $2^{r-1}-1$ non-empty subsets of $[r-1]$, the total number of unsigned permutations is $1+(2^{r-1}-1)=2^{r-1}$.  The lift count is exactly the distinction between \eqref{eq:recognition-w0}, \eqref{eq:recognition-c}, and \eqref{eq:recognition-d}.
\end{proof}

This recognition theorem is often the shortest route to applying the main result.  Instead of constructing $\Gamma_u$ for an arbitrary signed permutation $u$, one checks whether the unsigned permutation is an increasing cycle through the distinguished coordinate $r$, and then checks whether the sign vector is one of the two allowed vectors.  The exclusion lemmas in Section~\ref{sec:rigidity} prove the corresponding local statement: every forbidden simple left multiplication from a vertex of $\calF_r$ is detected by a loop or a two-cycle in the root-poset graph.  The global recognition criterion itself follows from the classification theorem.

\subsection{Lengths and defect distribution}
\label{sec:lengths}

Although the classification proof above does not require closed length formulas for $c_I$ and $d_I$, the formulas clarify the shape of the graph and provide a useful check on the induction.  We include them because they make the subset parametrization completely transparent.

For $I=\{i_1<\cdots<i_k\}\neq\varnothing$, define the gap lengths
$g_a=n_I(i_a)-i_a-1 \qquad(1\leq a\leq k),$
where $n_I(i_k)=r$.  The gaps start at $i_1$, not at $1$, and therefore
$g_a=\#\{b:i_a<b<n_I(i_a)\}, \qquad \sum_{a=1}^k g_a=r-i_1-k.$
The set $A_I$ has cardinality
$\#A_I=\sum_{a=1}^k g_a=r-i_1-k,$
whereas
$\#B_I=r-i_1.$
Hence
$\#\nu_0(c_I)=\#A_I+\#B_I=2r-2i_1-k.$
Since $\ell(u)=\#\{\beta\in\Pi_+:u(\beta)\in\Pi_-\}$ and $\#\Pi_+=r(r-1)$, this gives
\begin{equation}
  \ell(c_I)=r(r-1)-(2r-2i_1-k).
  \label{eq:length-cI}
\end{equation}
The same formula holds for $d_I$ by $\tau$-symmetry:
$\ell(d_I)=\ell(c_I).$
For $I=\{1\}$, one obtains
$\ell(c_{\{1\}})=\ell(d_{\{1\}})=r(r-1)-(2r-3),$
which agrees with Voloshyn's formula
$\frac{1}{2}\bigl(r(r-1)+(r-2)(r-3)\bigr) =r(r-1)-(2r-3)$
for $r$ odd.

It is cleaner to discuss length changes in terms of the defect
\begin{equation}
  \delta(I)=r(r-1)-\ell(c_I)=2r-2\min I-\#I
  \qquad(I\neq\varnothing),
  \label{eq:defect-I}
\end{equation}
with the convention $\delta(\varnothing)=0$ for $w_0$.  If $J=I\symdiff\{a\}$ is an admissible ordinary toggle, then $\ell(c_J)-\ell(c_I)=\delta(I)-\delta(J)=2(\min J-\min I)+(\#J-\#I)$; the same formula holds for $d_I$.  This identity records all subcases: toggling the minimum has the opposite length effect from toggling an interior point.  For the spin toggle $a=r-1$ one obtains
\begin{equation*}
  \ell(c_{I\symdiff\{r-1\}})-\ell(c_I)
  =\begin{cases}
     -1, & I=\varnothing,\\[1mm]
     +1, & I\neq\varnothing\text{ and } r-1\notin I,\\[1mm]
     -1, & r-1\in I\text{ and }I\neq\{r-1\},\\[1mm]
     +1, & I=\{r-1\},
   \end{cases}
\end{equation*}
where the first and last lines use $c_\varnothing=w_0$.  The corresponding formulas for $d_I$ are identical.

The defect distribution has a closed form.  Let
$F_r(q)=\sum_{u\text{ rational}} q^{r(r-1)-\ell(u)}.$
Then Theorem~\ref{thm:main-classification} and \eqref{eq:defect-I} give
\begin{align}
  F_r(q)
  &=1+2\sum_{\varnothing\neq I\subseteq [r-1]}q^{2r-2\min I-\#I}
  \notag\\
  &=1+2\sum_{m=1}^{r-1}\sum_{J\subseteq\{m+1,\ldots,r-1\}}
       q^{2r-2m-1-\#J}
  \notag\\
  &=1+2\sum_{n=0}^{r-2}q^{n+1}(1+q)^n.
  \label{eq:defect-polynomial}
\end{align}
Evaluating at $q=1$ recovers the count,
$F_r(1)=1+2\sum_{n=0}^{r-2}2^n=2^r-1.$
Thus the formula refines the cardinality by Coxeter length.

\begin{corollary}[length layers]
\label{cor:length-layers}
Let $N=r(r-1)$ and define
\begin{equation}
  a_t=\sum_{n=\lceil (t-1)/2\rceil}^{\min(t-1,r-2)}
       \binom{n}{t-n-1}
  \qquad(t\geq 1),
  \label{eq:at-def}
\end{equation}
with the convention that an empty sum is zero.  Then
\begin{equation}
  \#\{u\text{ rational}:N-\ell(u)=t\}
  =\begin{cases}
     1,&t=0,\\[1mm]
     2a_t,&t\geq1.
   \end{cases}
  \label{eq:length-layer-size}
\end{equation}
In particular, the possible nonzero defects are precisely $1,2,\ldots,2r-3$.
\end{corollary}

\begin{proof}
In \eqref{eq:defect-polynomial}, the summand $q^{n+1}(1+q)^n$ contributes
$\binom{n}{j}q^{n+1+j} \qquad(0\leq j\leq n).$
Thus the coefficient of $q^t$ in the one-half polynomial
$P_r(q)=\sum_{\varnothing\neq I\subseteq [r-1]}q^{N-\ell(c_I)}$
is the sum of $\binom{n}{t-n-1}$ over all $n$ such that $0\leq t-n-1\leq n$ and $0\leq n\leq r-2$.  These inequalities are exactly the bounds in~\eqref{eq:at-def}.  The factor $2$ in~\eqref{eq:length-layer-size} comes from the two lifts $c_I,d_I$, while the unique defect-zero element is $w_0$.  For fixed $n$, the polynomial $q^{n+1}(1+q)^n$ has strictly positive coefficients in every degree from $n+1$ through $2n+1$.  Hence its support is the integer interval $[n+1,2n+1]$.  As $n$ runs from $0$ to $r-2$, these intervals have no gaps: the interval for $n=0$ is $[1,1]$, and for $n\geq1$ the interval $[n+1,2n+1]$ begins no later than one integer after the preceding endpoint $2n-1$, because $n+1\leq2n$.  Their union is therefore $[1,2r-3]\cap\mathbb Z$, proving the asserted defect range.  The subsets $I=\{r-1\}$ and $I=\{1\}$ realize the two endpoints.
\end{proof}

\begin{corollary}[extreme rational elements]
\label{cor:extreme-lengths}
Among rational elements of $W(D_r)$, the unique element of length $N$ is $w_0$.  The rational elements of length $N-1$ are exactly
$c_{\{r-1\}}, \qquad d_{\{r-1\}},$
and the rational elements of minimal length are exactly
$c_{\{1\}}, \qquad d_{\{1\}}, \qquad \ell(c_{\{1\}})=\ell(d_{\{1\}})=N-(2r-3).$
\end{corollary}

\begin{proof}
The defect is $0$ only for $w_0$.  For non-empty $I$, equation~\eqref{eq:defect-I} gives
$N-\ell(c_I)=2r-2\min I-\#I.$
This is equal to $1$ only when $\min I=r-1$ and $\#I=1$, namely $I=\{r-1\}$.  It is maximal when $\min I=1$ and $\#I=1$, namely $I=\{1\}$.  The same argument applies to the $d$-half.
\end{proof}

For $r=5$, the defect polynomial is $F_5(q)=1+2\bigl(q+q^2+2q^3+3q^4+4q^5+3q^6+q^7\bigr)$.  Thus the length distribution, beginning from $N=20$, is
\begin{equation*}
\begin{array}{c|cccccccc}
 N-\ell&0&1&2&3&4&5&6&7\\
 \hline
 \#\{u\text{ rational}:N-\ell(u)=t\}&1&2&2&4&6&8&6&2.
\end{array}
\end{equation*}
This is the refined form of the count $31$.

\begin{example}
\label{ex:D5-lengths}
For $D_5$, formula~\eqref{eq:length-cI} gives, for instance,
\begin{equation*}
\begin{array}{c|cccccccc}
 I&\{1\}&\{2\}&\{3\}&\{4\}&\{1,2\}&\{1,3\}&\{2,4\}&\{1,2,3,4\}\\
 \hline
 \ell(c_I)&13&15&17&19&14&14&16&16.
\end{array}
\end{equation*}
The complete subset data are displayed below; together with the $d$-half and $w_0$ of length $20$, they account for all $31$ rational elements in $D_5$.  In the column headed $p_I$ we write only the non-trivial cycle; omitted points are fixed.
\begin{equation*}
\begin{array}{c|c|c|c}
 I&p_I&\ell(c_I)=\ell(d_I)&\deg(c_I)=\deg(d_I)\\
 \hline
 \{1\}&(1\ 5)&13&1\\
 \{2\}&(2\ 5)&15&2\\
 \{3\}&(3\ 5)&17&2\\
 \{4\}&(4\ 5)&19&2\\
 \{1,2\}&(1\ 2\ 5)&14&2\\
 \{1,3\}&(1\ 3\ 5)&14&2\\
 \{1,4\}&(1\ 4\ 5)&14&2\\
 \{2,3\}&(2\ 3\ 5)&16&3\\
 \{2,4\}&(2\ 4\ 5)&16&3\\
 \{3,4\}&(3\ 4\ 5)&18&3\\
 \{1,2,3\}&(1\ 2\ 3\ 5)&15&3\\
 \{1,2,4\}&(1\ 2\ 4\ 5)&15&3\\
 \{1,3,4\}&(1\ 3\ 4\ 5)&15&3\\
 \{2,3,4\}&(2\ 3\ 4\ 5)&17&4\\
 \{1,2,3,4\}&(1\ 2\ 3\ 4\ 5)&16&4.
\end{array}
\end{equation*}
For example, the row $I=\{2,3,4\}$ gives
$c_{\{2,3,4\}}=(-1,-3,-4,-5,2), \qquad d_{\{2,3,4\}}=(-1,-3,-4,5,-2),$
and these two vertices have degree four, with ordinary edges toggling $1,2,3$ whenever $2,3,4$ are present and one spin edge toggling $4$.
\end{example}

\subsection{The rationality graph as a subset graph}
\label{sec:subset-graph}

Theorem~\ref{thm:main-classification} gives a compact model for $\Gamma(D_r)$ that may be useful in applications to rational normal forms and cluster-type structures.  We spell it out as a purely combinatorial graph.

Let $\mathscr{B}_{r-1}$ be the graph with vertex set all subsets of $[r-1]$ and edges
\begin{equation}
  I\longleftrightarrow I\symdiff\{a\}
  \quad\text{if and only if}\quad
  \begin{cases}
    1\leq a\leq r-2\text{ and }a+1\in I,\\
    \text{or }a=r-1.
  \end{cases}
  \label{eq:B-graph}
\end{equation}
Then $\Gamma(D_r)$ is obtained from two copies of $\mathscr{B}_{r-1}$ by identifying only their empty-set vertices; in symbols, $\Gamma(D_r)\cong\mathscr{B}_{r-1}^{(c)}\cup_{\varnothing}\mathscr{B}_{r-1}^{(d)}$.
The edge labelled $a$ in the $c$-copy corresponds to left multiplication by $s_a$ for $a<r-1$ and by $s_{r-1}$ for $a=r-1$; in the $d$-copy it corresponds to $s_a$ for $a<r-1$ and by $s_r$ for $a=r-1$.

The graph $\mathscr{B}_{r-1}$ is connected.  We give a deletion algorithm because it is later useful for distance estimates.  Put $n=r-1$.  If $I\neq\varnothing$ and $m=\max I$, define the toggle sequence
\begin{equation}
  \omega_m=
  \begin{cases}
    (n),&m=n,\\
    (n,n-1,\ldots,m,m+1,\ldots,n),&m<n,
  \end{cases}
  \label{eq:omega-m}
\end{equation}
and apply its labels from left to right.  This sequence removes $m$ and leaves all elements below $m$ unchanged.  Indeed, because $m$ is the maximum of $I$, the descending segment first adds $n,n-1,\ldots,m+1$ one at a time; each addition is allowed because the next larger entry has just been added.  The label $m$ then removes $m$, and the ascending segment removes the temporary entries $m+1,\ldots,n$ in that order.  In formulas, for $m<n$,
\begin{equation}
  I
  \xleftrightarrow{n} I\cup\{n\}
  \xleftrightarrow{n-1} I\cup\{n-1,n\}
  \xleftrightarrow{n-2}\cdots
  \xleftrightarrow{m} (I\setminus\{m\})\cup\{m+1,\ldots,n\}
  \label{eq:delete-first-half}
\end{equation}
followed by
\begin{equation}
  (I\setminus\{m\})\cup\{m+1,\ldots,n\}
  \xleftrightarrow{m+1}
  (I\setminus\{m\})\cup\{m+2,\ldots,n\}
  \xleftrightarrow{m+2}\cdots
  \xleftrightarrow{n} I\setminus\{m\}.
  \label{eq:delete-second-half}
\end{equation}
Every toggle displayed in~\eqref{eq:delete-first-half}--\eqref{eq:delete-second-half} is allowed by \eqref{eq:B-graph}.  Repeating the procedure for the successive maxima of $I$ reaches $\varnothing$.  This gives a purely combinatorial proof, in type $D_r$ odd, of the connectedness part of Voloshyn's theorem.

The degree formula~\eqref{eq:degree-formula} becomes $\deg_{\mathscr{B}_{r-1}}(I)=1+\#(I\setminus\{1\})$ for $I\neq\varnothing$, while $\deg_{\mathscr{B}_{r-1}}(\varnothing)=1$.  After the two copies are glued at $\varnothing$, the common vertex has degree $2$.  Hence $\sum_{v\in\Gamma(D_r)}\deg(v)=2\sum_{I\subseteq[r-1]}\deg_{\mathscr{B}_{r-1}}(I)$, where the empty-set degree is counted once in each copy before gluing.  A direct summation gives
\begin{align*}
  \#E(\mathscr{B}_{r-1})
  &=\frac{1}{2}\left(1+\sum_{\varnothing\neq I\subseteq [r-1]}(1+\#(I\setminus\{1\}))\right)
  \\
  &=\frac{1}{2}\left(1+(2^{r-1}-1)+(r-2)2^{r-2}\right)
  \\
  &=2^{r-2}+(r-2)2^{r-3}.
\end{align*}
Therefore $\#E(\Gamma(D_r))=2^{r-1}+(r-2)2^{r-2}$.
For instance, $\#E(\Gamma(D_5))=40$; the corresponding degree distribution has total degree $2\cdot 1+13\cdot 2+12\cdot 3+4\cdot 4=80$.

The local neighbourhoods can be read without reference to the ambient Weyl group.  For $I\neq\varnothing$, put
$P(I)=\{a\in [r-2]:a+1\in I\}, \qquad Q(I)=P(I)\cup\{r-1\}.$
Then the neighbours of $c_I$ in $\Gamma(D_r)$ are
$\mathcal{N}(c_I)=\{c_{I\symdiff\{a\}}:a\in Q(I)\},$
where $c_\varnothing=w_0$, and similarly
$\mathcal{N}(d_I)=\{d_{I\symdiff\{a\}}:a\in Q(I)\}.$
The neighbours of the glued vertex are
$\mathcal{N}(w_0)=\{c_{\{r-1\}},d_{\{r-1\}}\}.$
Consequently, the graph distance from $c_I$ to $w_0$ is bounded by the total length of the deletion sequences in~\eqref{eq:omega-m}.  If $I=\{i_1<\cdots<i_k\}$ and $n=r-1$, apply $\omega_{i_k}$ first, then $\omega_{i_{k-1}}$, and continue in decreasing order of the indices until $\omega_{i_1}$.  After the $j$th deletion, the temporary entries introduced by that sequence have been removed, so the next prescribed index is again the maximum of the current subset.  The concatenated sequence therefore connects $I$ to $\varnothing$ in $\mathscr{B}_{r-1}$.  Hence
\begin{equation*}
  \operatorname{dist}_{\Gamma(D_r)}(c_I,w_0)
  \leq
  \sum_{j=1}^k\bigl(2(r-1-i_j)+1\bigr),
\end{equation*}
with the identical estimate for $d_I$.  In particular, because any path from a $c$-vertex to a $d$-vertex passes through $w_0$, one has
\begin{equation*}
  \operatorname{dist}_{\Gamma(D_r)}(c_I,d_J)
  =\operatorname{dist}_{\Gamma(D_r)}(c_I,w_0)
   +\operatorname{dist}_{\Gamma(D_r)}(w_0,d_J).
\end{equation*}
The exact diameter is not needed for the classification; the point is that connectedness, degrees, and cross-half distances follow from a finite-state subset calculus rather than from a search in the Weyl group of order $2^{r-1}r!$.

\subsection{The $D_5$ model}
\label{sec:D5-model}

We record the smallest odd case in full.  This is not used in the proof, but it provides a compact verification of the abstract subset graph and of the internal arrow formula.  Figure~\ref{fig:D5-paths} displays a labelled path in each half from $w_0$ to the corresponding valency-one vertex.

\begin{figure}[pos=t]
\centering
\begin{tikzpicture}[xscale=1.0,yscale=.88, every node/.style={font=\scriptsize}]
  \node (w0) at (0,0) {$w_0$};
  \node (c4) at (-1.4,-.9) {$c_4$};
  \node (c34) at (-2.25,-1.8) {$c_{34}$};
  \node (c234) at (-3.05,-2.7) {$c_{234}$};
  \node (c1234) at (-2.15,-3.6) {$c_{1234}$};
  \node (c134) at (-1.15,-4.5) {$c_{134}$};
  \node (c14) at (-.65,-5.4) {$c_{14}$};
  \node (c1) at (-.25,-6.3) {$c_1$};

  \node (d4) at (1.4,-.9) {$d_4$};
  \node (d34) at (2.25,-1.8) {$d_{34}$};
  \node (d234) at (3.05,-2.7) {$d_{234}$};
  \node (d1234) at (2.15,-3.6) {$d_{1234}$};
  \node (d134) at (1.15,-4.5) {$d_{134}$};
  \node (d14) at (.65,-5.4) {$d_{14}$};
  \node (d1) at (.25,-6.3) {$d_1$};

  \draw (w0)--node[left]{$4$}(c4);
  \draw (c4)--node[left]{$3$}(c34);
  \draw (c34)--node[left]{$2$}(c234);
  \draw (c234)--node[left]{$1$}(c1234);
  \draw (c1234)--node[left]{$2$}(c134);
  \draw (c134)--node[left]{$3$}(c14);
  \draw (c14)--node[left]{$4$}(c1);

  \draw (w0)--node[right]{$5$}(d4);
  \draw (d4)--node[right]{$3$}(d34);
  \draw (d34)--node[right]{$2$}(d234);
  \draw (d234)--node[right]{$1$}(d1234);
  \draw (d1234)--node[right]{$2$}(d134);
  \draw (d134)--node[right]{$3$}(d14);
  \draw (d14)--node[right]{$5$}(d1);
\end{tikzpicture}
\caption{A labelled subdiagram of $\Gamma(D_5)$ containing paths from $w_0$ to the two leaves.  We abbreviate $c_{\{i_1,\ldots,i_k\}}$ and $d_{\{i_1,\ldots,i_k\}}$ to $c_{i_1\cdots i_k}$ and $d_{i_1\cdots i_k}$.  Edge labels are simple reflections; the omitted vertices and edges are determined by the subset-toggle rules.}
\label{fig:D5-paths}
\end{figure}

Put $n=4$.  In the abstract graph $\mathscr{B}_4$, the edges with label $4$ are present for every subset because the spin toggle is always allowed:
\begin{align*}
 E_4=\bigl\{
 &\varnothing\!-\!4,
  1\!-\!14,
  2\!-\!24,
  3\!-\!34,
  12\!-\!124,
  13\!-\!134,
  23\!-\!234,
  123\!-\!1234
 \bigr\}.
\end{align*}
Here, for instance, $14$ means $\{1,4\}$.  The ordinary label-$1$ edges are exactly those pairs in which $2$ is present:
\begin{equation*}
  E_1=\{2\!-\!12,
        23\!-\!123,
        24\!-\!124,
        234\!-\!1234\}.
\end{equation*}
Similarly,
\begin{equation*}
  E_2=\{3\!-\!23,
        13\!-\!123,
        34\!-\!234,
        134\!-\!1234\},
\end{equation*}
where the condition is the presence of $3$, and
\begin{equation*}
  E_3=\{4\!-\!34,
        14\!-\!134,
        24\!-\!234,
        124\!-\!1234\},
\end{equation*}
where the condition is the presence of $4$.  Therefore
$\,\#E(\mathscr{B}_4)=\#E_4+\#E_1+\#E_2+\#E_3=8+4+4+4=20.$
The graph $\Gamma(D_5)$ consists of a $c$-copy and a $d$-copy of this graph, glued at $\varnothing=w_0$.  In the $c$-copy the abstract label $4$ is the simple reflection $s_4=s_{r-1}$, while in the $d$-copy it is $s_5=s_r$; the labels $1,2,3$ are the same in both copies.  Thus
$\#V(\Gamma(D_5))=1+2(2^4-1)=31, \qquad \#E(\Gamma(D_5))=2\cdot20=40.$

The internal root-poset graphs for individual vertices also have the predicted two-level form.  For $I=\{1,3\}$, one has $p_I=(1\ 3\ 5)$, and Proposition~\ref{prop:nu0-cI} gives
\begin{align*}
  A_{13}&=\{e_2-e_3,
            e_4-e_5\},\\
  B_{13}&=\{e_1-e_2,
            e_1-e_3,
            e_1-e_4,
            e_1-e_5\}.
\end{align*}
By Proposition~\ref{prop:exact-arrows-cI}, the outgoing arrows are
\begin{align*}
  e_2-e_3&\longrightarrow
    e_1-e_2,
    e_1-e_3,
    e_1-e_4,
    e_1-e_5,\\
  e_4-e_5&\longrightarrow
    e_1-e_4,
    e_1-e_5.
\end{align*}
No arrow starts from any element of $B_{13}$.  Thus $\nu_1(c_{13})=A_{13}$ and $\nu_2(c_{13})=\varnothing$.

For $I=\{2,4\}$, the cycle is $p_I=(2\ 4\ 5)$ and
$A_{24}=\{e_3-e_4\}, \qquad B_{24}=\{e_2-e_3,e_2-e_4,e_2-e_5\}.$
The unique $A$-vertex points to all three $B$-vertices:
$e_3-e_4\longrightarrow e_2-e_3, \,e_2-e_4, \,e_2-e_5.$
For the full subset $I=\{1,2,3,4\}$, all gaps have length zero, so
\begin{equation*}
  A_{1234}=\varnothing,
  \qquad
  B_{1234}=\{e_1-e_2,e_1-e_3,e_1-e_4,e_1-e_5\},
  \qquad
  \nu_1(c_{1234})=\varnothing.
\end{equation*}
These three examples illustrate configurations with two, one, and zero vertices in $A_I$, respectively; they are not an exhaustive list of the possible internal graphs in $D_5$.  Applying $\tau$ gives the corresponding $d$-vertices without changing the abstract two-level structure.

\section{Rational normal forms and concluding remarks}
\label{sec:normal-form-consequences}

Voloshyn proves that every rational Weyl group element, with any chosen representative, gives a solution of the rational-decomposition problem \cite[Theorem~2.7]{Voloshyn2026}.  Consequently, for every rational $u$ in Theorem~\ref{thm:main-classification} and every chosen representative $\overline u\in N_G(H)$, his construction yields rational maps $N:G\dashrightarrow N_-$ and $B:G\dashrightarrow B_+$ satisfying \eqref{eq:intro-normal-form}.  The classification therefore gives the complete list of rational Weyl-group indices in odd type $D$ for which this construction is guaranteed:
$w_0, \qquad c_I,\ d_I \quad(\varnothing\neq I\subseteq[r-1]).$
There are exactly $2^r-1$ such indices.  This is a count of Weyl-group indices, or equivalently of normal-form types after one representative has been fixed for each index; it is not a count of all representatives in $N_G(H)$.  Moreover, the statement uses only the implication ``rational element $\Rightarrow$ solution'' and does not assert the converse for arbitrary Weyl group elements.

The two-level calculation also gives a uniform bound on the stabilization depth in Voloshyn's recursive construction.  We record the precise statement, including the parabolic induction needed to pass from the root sequence to the recursive approximations.

\begin{proposition}[Two-step stabilization]
\label{prop:two-step-stabilization}
Let $u=c_I$ or $u=d_I$ with $\varnothing\neq I\subseteq[r-1]$.  Start Voloshyn's recursion from the longest-element solution and choose the special representatives used in \cite[Proposition~2.9]{Voloshyn2026}.  Then the approximations $P_k$ and $N_k$ stabilize by stage $2$.  If $A_I=\varnothing$, they stabilize by stage $1$.
\end{proposition}

\begin{proof}
For both families, Proposition~\ref{prop:exact-arrows-cI} and its $\tau$-image give
\begin{equation}
  \nu_0(c_I)\supseteq\nu_1(c_I)=A_I\supseteq\nu_2(c_I)=\varnothing,
  \qquad
  \nu_0(d_I)\supseteq\nu_1(d_I)=\tau(A_I)\supseteq\nu_2(d_I)=\varnothing.
  \label{eq:stabilization-depth}
\end{equation}
Take $v=w_0$ in the relative construction of \cite[Definition~2.8 and Proposition~2.9]{Voloshyn2026}.  Since $w_0^{-1}(\Pi_-)=\Pi_+$, the relative sequence $\nu^k(u\mid w_0)$ is exactly the sequence $\nu_k(u)$ used in this paper.  Let $\widetilde B$ and $\widetilde N$ denote the rational maps for the longest-element solution.  With the special representatives fixed in \cite[Section~3.3]{Voloshyn2026}, the recursion is
\begin{align*}
  P_0(g)&=\widetilde B(g)\overline{w_0}\,\overline u^{-1},\\
  P_k(g)&=[P_{k-1}(g)]_{\oplus}\,\overline u [P_{k-1}(g)]_-\overline u^{-1}\quad(k\geq1),\\
  N_k(g)&=\widetilde N(g)\widetilde N(P_k(g)\overline u)^{-1}.
\end{align*}
Here $x=[x]_- [x]_{\oplus}$ denotes the Gaussian decomposition on the big cell.

For $k\geq0$, set $\Delta^k=\Delta\cap\Adj(\nu_k(u))$, and let $\mathcal P_k$ be the standard parabolic subgroup generated by $B_+$ and the negative simple-root subgroups $U_{-\alpha}$ with $\alpha\in\Delta^k$.  We claim that $P_k(g)\in\mathcal P_k$ for generic $g$ and every $k\geq0$.

For the base step, choose a reduced decomposition $uw_0^{-1}=s_{i_1}\cdots s_{i_m}$.  Proposition~3.9 of \cite{Voloshyn2026} implies $\alpha_{i_j}\in\Delta^0$ for every $j$.  For the special representatives above, the factorization used in the proof of \cite[Proposition~2.9]{Voloshyn2026} writes $P_0(g)$ as a product of elements of $B_+$ and negative simple-root elements $y_{i_j}(1)\in U_{-\alpha_{i_j}}$.  Every factor therefore lies in $\mathcal P_0$, and hence $P_0(g)\in\mathcal P_0$.

Assume inductively that $P_{k-1}(g)\in\mathcal P_{k-1}$.  On the big cell of the standard parabolic $\mathcal P_{k-1}$, uniqueness of Gaussian decomposition gives $[P_{k-1}(g)]_-\in N_-\cap\mathcal P_{k-1}$.  As an algebraic subgroup, $N_-\cap\mathcal P_{k-1}$ is generated by the simple negative root subgroups $U_{-\alpha}$ with $\alpha\in\Delta^{k-1}$; hence $[P_{k-1}(g)]_-$ may be written as a finite product of elements from these subgroups.  It is therefore enough to consider one factor in $U_{-\alpha}$ with $\alpha\in\Delta^{k-1}$.  Conjugation by $\overline u$ sends it to $U_{-u(\alpha)}$.  If $u(\alpha)<0$, this is a positive-root subgroup and is contained in $B_+\subseteq\mathcal P_k$.  If $u(\alpha)>0$, then $\alpha\in\Adj(\nu_{k-1}(u))$ gives $u(\alpha)\in\nu_k(u)$.  Every simple root in the support of $u(\alpha)$ lies below $u(\alpha)$ and hence belongs to $\Delta^k$.  The standard parabolic $\mathcal P_k$ contains the negative root subgroup associated with every positive root supported on $\Delta^k$; consequently $U_{-u(\alpha)}\subseteq\mathcal P_k$.  Since $[P_{k-1}(g)]_{\oplus}\in B_+\subseteq\mathcal P_k$, the recursion yields $P_k(g)\in\mathcal P_k$.  This proves the claim.

Equation~\eqref{eq:stabilization-depth} gives $\Delta^2=\varnothing$, so $\mathcal P_2=B_+$ and $P_2(g)\in B_+$ generically.  Corollary~3.19 of \cite{Voloshyn2026} then implies $P_k=P_2$ for all $k\geq2$.  The displayed formula for $N_k$ shows simultaneously that $N_k=N_2$ for all $k\geq2$.  If $A_I=\varnothing$, then $\nu_1(u)=\varnothing$, so the same argument applies with $1$ in place of $2$.
\end{proof}

Thus the special recursive implementation requires at most two correction steps for every non-longest rational element in the classification.  This stage bound concerns the special representatives and recursion above; the existence of the rational decomposition itself holds for every representative by \cite[Theorem~2.7]{Voloshyn2026}.  At the combinatorial level, the bound is the assertion that neither $\Gamma_{c_I}$ nor $\Gamma_{d_I}$ admits a directed walk of length two.

\subsection{Conclusion}
\label{sec:conclusion}

We have proved that, for odd $r\geq 5$, the rational Weyl group elements of type $D_r$ are exactly
$w_0, \qquad c_I,d_I\quad(\varnothing\neq I\subseteq [r-1]),$
where $c_I,d_I$ are the two signed cyclic elements defined in~\eqref{eq:cI-intro}--\eqref{eq:dI-intro}.  This proves the closed formula
$\#\{u\in W(D_r):u\text{ rational}\}=2^r-1$
and determines the entire rationality graph.  The graph is the union of two subset graphs glued at $w_0$, and its two leaves are precisely
\begin{align*}
  c_{\{1\}}&:\quad e_1\mapsto -e_r,\quad e_j\mapsto -e_j\ (2\leq j\leq r-1),\quad e_r\mapsto e_1,\\
  d_{\{1\}}&:\quad e_1\mapsto e_r,\quad e_j\mapsto -e_j\ (2\leq j\leq r-1),\quad e_r\mapsto -e_1.
\end{align*}
The classification reduces the odd type-$D$ rationality problem to a rigid and explicit subset calculus.  The more irregular type-$A$ enumerations in \cite{Voloshyn2026} indicate that any analogous description there would require additional combinatorial structure.  In odd type $D$, the spin-node automorphism explains the symmetry between the two halves; the one-step rigidity theorem and the descent induction, rather than the automorphism alone, prove that the two halves exhaust the rationality graph and intersect only at $w_0$.

\section*{Funding}
No funding was received for this work.

\section*{Declaration of competing interest}
The authors declare no competing interests.

\section*{Data availability}
No datasets were generated or analyzed in this work.
\section*{Declaration of Generative AI and AI-Assisted Technologies in the Writing Process}
During the preparation of this work, the authors used DeepSeek to build a specialized agent for solving mathematical problems, which was employed to generate an initial proof of the main theorem. After using this tool, the authors reviewed and edited the content as needed and take full responsibility for the content of the manuscript.

\end{document}